\documentclass{amsart}

\usepackage{amscd}
\usepackage{amssymb}
\usepackage{soul}
\usepackage{mathrsfs}

\usepackage[all]{xypic}
\usepackage[dvipsnames] {xcolor}
\usepackage{tikz}
\usepackage{enumitem}

\numberwithin{equation}{section}

\theoremstyle{plain} 
\newtheorem{theorem}[equation]{Theorem}
\newtheorem{lemma}[equation]{Lemma}

\newtheorem{proposition}[equation]{Proposition}
\newtheorem{corollary}[equation]{Corollary}

\newcommand{\argforcustom}{}
\theoremstyle{newtextthm}
\newtheorem{helperforcustom}[equation]{\argforcustom}
\newtheorem*{helperforcustomstar}{\argforcustom}

\newenvironment{customstar}[1]{\renewcommand{\argforcustom}{#1}\begin{helperforcustomstar}}{\end{helperforcustomstar}}

\theoremstyle{definition}
\newtheorem{definition}[equation]{Definition}

\newtheorem{example}[equation]{Example}

\newtheorem{remark}[equation]{Remark}
\newtheorem*{remark*}{Remark}

\newcommand{\defining}[1]{{\emph{#1}}}
\newcommand{\definedas}{:=}

\newcommand{\itemref}[1]{(\ref{#1})}

\newcommand{\integers}{{\mathbb{Z}}}
\newcommand{\naturals}{{\mathbb{N}}}

\newcommand{\complexes}{{\mathbb{C}}}
\newcommand{\field}{{\mathbb{F}}}

\newcommand{\Uofp}{\U(p)}
\newcommand{\SUofp}{{\SU(p)}}
\newcommand{\Borel}{\SLUpper}
\newcommand{\BorelSUp}{\SLUpper}
\newcommand{\collection}{{\mathscr{C}}}
\newcommand{\rescen}{{\rm{R}}}

\def\doCal#1{%
\ifx#1\doAllCalEnd\def\doAllCal{\relax}\else%
 \expandafter\edef\csname#1cal\endcsname{{\noexpand\mathcal #1}}\fi}
\def\doAllCal#1{\doCal#1\doAllCal}
\doAllCal ABCDEFGHIJHKLMNOPQRSTUVWXYZ\doAllCalEnd

\def\doBar#1{%
\ifx#1\doAllBarEnd\def\doAllBar{\relax}\else%
 \expandafter\edef\csname#1bar\endcsname{{\noexpand\overline{#1}}}\fi}
\def\doAllBar#1{\doBar#1\doAllBar}
\doAllBar ABCDEFGHIJHLMNOPQRSTUVWXYZabcdefghijhlmnopqrstuvwxyz\doAllBarEnd
\def\doWiggle#1{%
\ifx#1\doAllWiggleEnd\def\doAllWiggle{\relax}\else%
 \expandafter\edef\csname#1wiggle\endcsname{{\noexpand\tilde{#1}}}\fi}
\def\doAllWiggle#1{\doWiggle#1\doAllWiggle}
\doAllWiggle ABCDEFGHIJHLMNOPQRSTUVWXYZabcdefghijklmnopqrstuvwxyz\doAllWiggleEnd
\newcommand{\phiwiggle}{\strut\widetilde{\phi}}
\newcommand{\psiwiggle}{\strut\widetilde{\psi}}

\newcommand{\iotawiggle}{\tilde{\iota}}

\def\doBold#1{%
\ifx#1\doAllBoldEnd\def\doAllBold{\relax}\else%
 \expandafter\edef\csname#1bold\endcsname{{\noexpand\bf #1}}\fi}
\def\doAllBold#1{\doBold#1\doAllBold}
\doAllBold ABCDEFGHIJHKLMNOPQRSTUVWXYZabcdefghijklmnopqrstuvwxyz\doAllBoldEnd

\def\doBbb#1{%
\ifx#1\doAllBbbEnd\def\doAllBbb{\relax}\else%
 \expandafter\edef\csname#1bb\endcsname{{\noexpand\mathbb{#1}}}\fi}
\def\doAllBbb#1{\doBbb#1\doAllBbb}
\doAllBbb ABCDEFGHIJHKLMNOPQRSTUVWXYZ\doAllBbbEnd

\newcommand{\Largestrut}{\mbox{{\Large\strut}}}
\newcommand{\largestrut}{\mbox{{\large\strut}}}

\makeatletter
\newcommand{\switchmargin}{
\if@reversemargin
\normalmarginpar
\else
\reversemarginpar
\fi
}
\makeatother

\definecolor{llgray}{RGB}{230,230,230}
\definecolor{lblue}{RGB}{217,249,237}
\newcommand{\highlight}[1]{\ifmmode{\text{\sethlcolor{llgray}\hl{$#1$}}}\else{\sethlcolor{llgray}\hl{#1}}\fi}
\newcommand{\highlighteva}[1]{\ifmmode{\text{\sethlcolor{llgray}\hl{$#1$}}}\else{\sethlcolor{lblue}\hl{#1}}\fi}

\DeclareMathOperator{\Aff}{Aff}
\DeclareMathOperator{\Aut}{Aut}
\DeclareMathOperator{\B}{B}
\DeclareMathOperator{\BAut}{BAut}
\DeclareMathOperator{\BU}{BU}
\DeclareMathOperator{\BSU}{BSU}
\DeclareMathOperator{\BSO}{BSO}
\newcommand{\colim}{\operatorname{colim}\,}

\DeclareMathOperator{\GL}{GL}
\newcommand{\hocolim}{\operatorname{hocolim}}

\DeclareMathOperator{\Hom}{Hom}

\DeclareMathOperator{\ob}{ob}
\newcommand{\Id}{\mathrm{Id}}
\DeclareMathOperator{\im}{im}

\DeclareMathOperator{\mmod}{mod}

\DeclareMathOperator{\Out}{Out}

\DeclareMathOperator{\size}{size}
\DeclareMathOperator{\SL}{SL}

\DeclareMathOperator{\SU}{SU}
\DeclareMathOperator{\U}{U}

\newcommand{\epi}{\twoheadrightarrow}

\newcommand{\pdash}{$p$\kern1.3pt-}
\newcommand{\whatever}{\text{--}}

\newcommand{\Fcenrad}{\Fcal^{cr}}
\newcommand{\Lcenrad}{\Lcal^{cr}}

\newcommand{\Zpinfinity}{\integers/{p}^{\infty}}

\newcommand{\Zphat}{\widehat{\integers}_p}

\newcommand{\Linkwiggle}{\til{\Lcal}}

\newcommand{\twobytwo}[4]{
   \left(\begin{array}{cc}
   #1 & #2\\
   #3 & #4
   \end{array}\right)}

\newcommand{\Gammabold}{\mathbf{\Gamma}}

\newcommand{\inclusion}[2]{\iota\mkern1.5mu_{#1}^{#2}}

\newcommand{\subgroupeq}{\subseteq}
\newcommand{\subgroupneq}{\subsetneq}
\newcommand{\subgroup}{\subset}

\newcommand{\Slominska}{S{\l}omi\'{n}ska}
\newcommand{\Aguade}{Aguad\'e}
\newcommand{\AZ}{\Aguade--Zabrodsky }

\newcommand{\sd}{{\bar{s}d}}

\newcommand{\TboldSUp}{\Tbold_{\SU(p)}}

\newcommand{\SboldSUp}{\Sbold_{\SU(p)}}
\newcommand{\GammaSUp}{\Gamma}

\newcommand{\GLtwoFp}{\GL_2\!\field_p}
\newcommand{\SLtwoFp}{\SL_2\!\field_p}

\newcommand{\SLtwoFtwo}{\SL_2\!\field_2}
\DeclareMathOperator{\UTriangular}{\Ucal}
\newcommand{\GLUpper}{\UTriangular(\GLtwoFp)}
\newcommand{\SLUpper}{\UTriangular(\SLtwoFp)}
\newcommand{\SLUppertwo}{\UTriangular(\SLtwoFtwo)}

\newcommand{\xlongrightarrow}[1]{\xrightarrow{\,#1\,}}

\newcommand\restr[2]{{
  \left.\kern-\nulldelimiterspace 
  #1 
  \vphantom{\big|} 
  \right|_{#2} 
  }}

\newcommand{\suchthat}[1]{\left|\, #1 \right. }

\newcommand{\realize}[1]{\left|#1 \right|}
\newcommand{\pcomplete}[1]{{#1}_{p}^{\wedge}}

\renewcommand{\bar}{\widebar} 

\usepackage{hyperref}
\hypersetup{colorlinks=true,linkcolor=violet,citecolor=purple} 

\usepackage{marginnote}

\usepackage[nolabel]{showlabels} 
\definecolor{llgray}{RGB}{230,230,230}

\newcommand{\Z}{\integers}
\renewcommand{\L}{\Lcal}
\providecommand{\isom}{\cong}

\newcommand{\x}{\times}

\newcommand{\Sp}{\operatorname{Sp}}
\newcommand{\ints}{\cap}

\newcommand{\into}{\hookrightarrow}

\newcommand{\bto}{\leftarrow}
\newcommand{\hteq}{\simeq}
\renewcommand{\isom}{\cong}
\newcommand{\til}{\widetilde}

\newcommand{\Top}{\operatorname{Top}}
\newcommand{\widebar}[1]{{\overline{#1}}}
\newcommand{\F}{\mathcal{F}}

\newcommand{\longleftrightarrows}{\xymatrix@1@C=16pt{
\ar@<0.4ex>[r] & \ar@<0.4ex>[l]
}}

\renewcommand{\atop}[1]{{\let\scriptstyle\textstyle\let\scriptscriptstyle\scriptstyle\substack{#1}}}

\newcommand{\fusion}{$\Fcal$}

\newcommand{\AbstractDecompThmText}{Let $(S,\Fcal, \Lcal)$ be a \pdash local compact group. 
There is a functor
$\delta\colon \sd\Fcenrad\longrightarrow \Top$ with an
equivalence
$\hocolim_{\sd(\Fcenrad)} \delta \longrightarrow \left|\Lcal\right|$
and a natural homotopy equivalence
$\BAut_{\Lcal}(\Pbb)\longrightarrow \delta([\Pbb])$ for each chain $\Pbb$. 
Further, the group
$\Aut_{\Lcal}(\Pbb)$ is a virtually discrete \pdash toral group.
}

\newcommand{\NormalizerDecompLieText}{
Let $G$ be a compact Lie group and let $(S, \Fcal, \Lcal)$ be the associated \pdash local compact group.
 If $P\in\Fcenrad$, let $\Pbold$ denote its closure in~$G$. 
If $\Pbb=(P_0\subgroupneq\ldots\subgroupneq P_k)$ is a proper chain of subgroups in~$\Fcenrad$, then there is a natural weak mod~$p$ equivalence 
\begin{equation*}    \label{eq: BAut_G}  
\BAut_{\Lcal}(\Pbb)
\simeq
    \B\!\big(\ints_i N_G(\Pbold_i)\big).
\end{equation*}
If in addition $\pi_0G$ is a \pdash group, then  
the functor $\Fcenrad\rightarrow\Rcal$ given by $P\mapsto\Pbold$ induces an isomorphism of posets $\sd\Fcenrad\cong\sd\Rcal$.
}

\newcommand{\SUDecompositionTheoremText}{
Let $\SLUpper$ denote the group of upper triangular matrices 
in $\SLtwoFp$, let $T$ denote the chosen maximal discrete \pdash torus of~$\SUofp$, and let $\Gamma$ denote the extra-special \pdash group of order $p^3$ and exponent~$p$. 
\begin{enumerate}
\item 
For odd primes, the homotopy pushout of the diagram below is mod~$p$ equivalent to~$\BSU(p)$:
\[
\xymatrix@C=5pt{
\B\!\left(\GammaSUp\rtimes \BorelSUp\right)
     \ar[d]\ar[rr] 
&&
\B\!\left(\strut T\rtimes\Sigma_p\right)
\\
\B\!\left(\strut \GammaSUp\rtimes \SLtwoFp \right) 
.
}
\]
\item \cite[Thm.~4.1]{DMW1} 
The homotopy pushout of the diagram below is mod~$2$ equivalent to~$\BSU(2)$: 
\[
\xymatrix@C=5pt{
\B Q_{16} 
 \ar[rr]\ar[d] 
 &&   \B\!\left(\strut T\rtimes\Sigma_2\right)
\\\B O_{48},
}
\]
where $Q_{16}$ and $O_{48}$ denote the quaternionic group of order~$16$ and the binary octahedral group of order~$48$, respectively. 
\end{enumerate}
}

\newcommand{\AZtheoremdiagram}{
\xymatrix@C=5pt{
\B\!\big(\Gamma\rtimes\GLUpper\big)
      \ar[d]\ar[rr]
   &&\B\left(T\rtimes G\right)
\\
\B\!\left(\Gamma\rtimes\GLtwoFp\right)
}
}

\newcommand{\AZTheoremText}{
Let $X$ denote one of the \AZ \pdash compact groups 
$X_{12}$ (with $p=3$),
$X_{29}$ (with $p=5$),
$X_{31}$ (with $p=5$), 
or $X_{34}$ (with $p=7$).
Let $T\cong (\Zpinfinity)^{p-1}$ denote the maximal discrete \pdash torus 
in the associated fusion system, and let $G$ be the Weyl group associated to~$X$. 
The homotopy pushout of the diagram
\[
\AZtheoremdiagram
\]
is homotopy equivalent to the nerve of the linking system associated to~$X$, and mod~$p$ equivalent to $BX$ itself. } 

\begin{document}

\title{Normalizer decompositions of \pdash local compact groups }

\author{Eva Belmont}
\address{Department of Mathematics, University of California San Diego, La Jolla, CA, USA}
\email{ebelmont@ucsd.edu}

\author{Nat\`alia Castellana}
\address{Departament de Matem\`atiques, Universitat Aut\`onoma de Barcelona, and Centre de Recerca Matemàtica, Barcelona, Spain}
\email{natalia@mat.uab.cat}

\author{Jelena Grbi\'{c}}
\address{School of Mathematical Sciences, University of Southampton, Southampton, UK }
\email{j.grbic@soton.ac.uk}

\author{Kathryn Lesh}
\address{Department of Mathematics, Union College, Schenectady NY, USA}
\email{leshk@union.edu}

\author{Michelle Strumila}
\address{School of Mathematics, Monash University, Clayton, Victoria, Australia}
\email{michelle.strumila@gmail.com}



\begin{abstract}
We give a normalizer decomposition for a \pdash local compact group $(S,\Fcal,\Lcal)$ that describes $\realize{\Lcal}$ as a homotopy colimit indexed over a finite poset. Our work generalizes the normalizer decompositions for finite groups due to Dwyer, for \pdash local finite groups due to Libman, and for compact Lie groups in separate work due to Libman. Our approach gives a result in the Lie group case that avoids topological subtleties with Quillen's Theorem~A,
because we work with discrete groups.
We compute the normalizer decomposition for the \pdash completed classifying spaces of
$\Uofp$ and $\SUofp$ and for the \pdash compact groups of \Aguade\  and Zabrodsky. 
\end{abstract}

\maketitle

\markboth{\sc{Belmont, Castellana, Grbi\'{c}, Lesh, and Strumila}}
{\sc{Normalizer decompositions}}

\section{Introduction}

For a finite group~$G$, a prime~$p$, and a suitable collection 
$\collection$ of subgroups of~$G$, Dwyer~\cite{Dwyer-Homology} gave a systematic approach to three homotopy colimit decompositions for~$\pcomplete{BG}$. 
Two of them, the centralizer decomposition 
and the subgroup decomposition,
are indexed by fusion  and orbit categories whose objects are conjugacy classes of subgroups in~$\collection$. The third is the normalizer decomposition,
\[
\pcomplete{BG} \hteq 
     \pcomplete{\big[\hocolim_{\,\Pbb\in\Ical}\delta \big]},
\]
which is indexed over the finite poset $\Ical$ whose objects are
$G$-conjugacy classes of chains 
$\Pbb = (H_0\subgroupneq \ldots \subgroupneq H_k)$ for $H_i\in
\collection$ and $\delta(\Pbb ) \simeq B(\ints_i N_G H_i)$.

There are two more general contexts that are relevant to our work. First, the homotopy colimit decompositions were studied for compact Lie groups. Jackowski and McClure 
\cite[Thm.~1.3]{jackowski-mcclure} had given a centralizer
decomposition for compact Lie groups with respect to the collection of elementary abelian \pdash subgroups. Jackowski, McClure and Oliver described subgroup decompositions. \Slominska\  \cite{slominska-webb-conj} gave a normalizer decomposition with this
collection. Libman \cite[Thm.~C]{Libman-Minami} gave centralizer, subgroup,
and normalizer decompositions with the collection of abelian \pdash subgroups and
the collection of \pdash radical subgroups.

The second generalization of interest comes by replacing finite groups by \emph{\pdash local finite groups}. These objects are triples $(S,\Fcal,\Lcal)$, where $S$ is a finite \pdash group, and $\Fcal$ and $\Lcal$ are categories encoding data that mimics conjugations. 
One can define the classifying space 
$\pcomplete{\realize{\Lcal}}$ of a \pdash local finite
group, and it behaves similarly to the \pdash completed classifying space of a group.
Every finite group gives rise to a \pdash local finite group, but not every
\pdash local finite group comes from a group.
Libman \cite{Libman-normalizer} proved the existence of a normalizer decomposition for classifying spaces of \pdash local finite groups.

We work with \emph{\pdash local compact groups} (Definition~\ref{definition: p-local compact group}), which generalize 
compact Lie groups in the same way that \pdash local
finite groups generalize finite groups. 
The theory, developed in~\cite{BLO-Discrete}, is in spirit analogous to that of \pdash local finite groups, 
but with new challenges because of the non-finite context. The analogue of a finite \pdash group is a 
discrete \pdash toral group, namely an extension of a discrete torus $(\Z/p^\infty)^r$ by a
finite \pdash group. Studying Lie groups by way of their associated \pdash local
compact groups has the advantage of reducing to discrete (as opposed to
topological) groups. Broto, Levi, and Oliver prove a subgroup
decomposition for \pdash
local compact groups~(\cite[Prop.~4.6]{BLO-Discrete}, \cite[Thm.~B]{LL-ExistenceL}).

Our first contribution is to adapt Libman's work \cite{Libman-normalizer} on \pdash local finite groups
to construct a normalizer decomposition for \pdash local compact groups. One advantage of the normalizer decomposition of a \pdash local compact group 
over the centralizer or
subgroup decompositions is that the normalizer decomposition is indexed over a finite poset. In the statement below, the notation $\Fcenrad$ refers to the subcategory of $\Fcal$ whose objects are $\Fcal$-centric and $\Fcal$-radical subgroups. The finite poset $\sd\Fcenrad$ is defined in Definition~\ref{def:sdbar}. 

\begin{customstar}
{Theorem~\ref{theorem: abstract decomposition theorem}}
\AbstractDecompThmText
\end{customstar}

This statement is largely formal, as is most of the proof, but there are interesting challenges in computing the decomposition in cases of interest. Our first class of examples is given by \pdash local compact groups arising from compact Lie groups.
Theorem~\ref{theorem: abstract decomposition theorem} takes place fully in the world of discrete \pdash toral groups. To relate our underlying theory to 
the usual context of normalizers in compact Lie groups, we use our previous work~\cite{WIT-normalizers} to rephrase the functor values in the resulting homotopy colimit as mod~$p$ equivalent to classifying spaces of group-theoretic normalizers. In comparison to~\cite{Libman-Minami},
this approach gives a more formal proof of the basic
decomposition result---essentially analogous to the finite case---because we do not have to address the topological issues in applying Quillen's Theorem~A. Instead, the topological issues can be 
neatly packaged into the functor values and understood on a uniform basis \cite{WIT-normalizers}. 

 Let $\Rcal$ denote the collection of \pdash toral subgroups of a compact Lie group $G$ that are both \pdash centric and \pdash stubborn in~$G$, and let $\sd\Rcal$ be the poset of $G$-conjugacy classes of chains
of proper inclusions of subgroups in~$\Rcal$.

\begin{customstar}{Theorem~\ref{theorem: NormalizerDecompLie}}
\NormalizerDecompLieText 
\end{customstar}

In concert with Theorem~\ref{theorem: abstract decomposition theorem}, Theorem~\ref{theorem: NormalizerDecompLie} tells us that when $\pi_0 G$ is a \pdash group, we can compute a normalizer decomposition
for $\pcomplete{BG}$ over a poset indexed by chains of \pdash centric and \pdash stubborn subgroups of~$G$, with values that are mod~$p$ equivalent to intersections of normalizers of those subgroups. 
(See Remark~\ref{remark: pi0 hypothesis} regarding the $\pi_0$ hypothesis.)
In Section~\ref{sec:U(p) SU(p)}, we compute the decomposition explicitly in the cases $\U(p)$ and $\SU(p)$, 
expressing the classifying spaces of these groups as mod~$p$ equivalent to a homotopy pushout diagram.
We believe these decompositions are new for odd primes. In the case 
$p=2$ we recover the theorem of Dwyer,
Miller, and Wilkerson \cite{DMW1}, who gave mod~$2$ homotopy pushout
decompositions of $\BSU(2)$ and $\BSO(3)$ using an ad hoc method.

\begin{customstar}
{Theorem~\ref{theorem: SU(p) decomposition}}
\SUDecompositionTheoremText
\end{customstar}

The result for $\Uofp$ is found in Theorem~\ref{theorem: U(p) decomposition}. 

For our second class of examples, we turn to \pdash compact groups, another generalization of classifying spaces of Lie groups introduced by Dwyer and
Wilkerson~\cite{dwyer-wilkerson}, which arose in the study of cohomology rings of loop spaces. Every \pdash compact group has an associated \pdash local compact
group $(S,\F,\L)$~\cite[Thm.~10.7]{BLO-Discrete}, but not every \pdash local compact group arises in this way. 
The classification of \pdash compact
groups in \cite{Andersen-Grodal-2compact, AGMV-pcompact} builds these spaces out of compact Lie groups and a collection of exotic examples. 

We apply Theorem~\ref{theorem: abstract decomposition theorem} to the \AZ \pdash compact groups, which are
closely related to $\pcomplete{\BSU(p)}$.
They were 
first constructed in \cite{aguade-modular} to have cohomology realizing certain
invariants of polynomial rings. Our result is a 
homotopy pushout diagram for these
spaces.

\begin{customstar}
{Theorem~\ref{theorem: pushout for AZ}}
\AZTheoremText
\end{customstar}

Comparing to Theorem~\ref{theorem: SU(p) decomposition}, above, we see that $\SLtwoFp$ is replaced by~$\GLtwoFp$, and $\Sigma_p$ (the Weyl group for~$\SUofp$) is replaced by~$G$ (the new, enlarged Weyl group for~$X$).

\subsection*{Organization} 
In Section~\ref{section: fusion systems} we
review the properties of fusion and linking systems in the setting of discrete \pdash local compact groups. In
Section~\ref{sec:abstract-result} we establish the general normalizer decomposition for a \pdash local compact group (Theorem~\ref{theorem: abstract decomposition theorem}), and we show that the spaces involved are classifying spaces of virtually discrete \pdash toral groups.
In Section~\ref{sec:Lie2}, we turn to Lie groups 
and prove Theorem~\ref{theorem: NormalizerDecompLie}.
The main issue to be addressed is the calculation of automorphisms in the linking system associated to a Lie group. The problem is that the model of the linking system associated to $G$ is not directly related to the transporter system. 
In Section~\ref{section: boring group theory} we 
prepare for application of the normalizer decompositions to $\Uofp$ and $\SUofp$ by separating out some group-theoretic calculations. In Section~\ref{sec:U(p) SU(p)} we use those calculations, in conjunction with Theorem~\ref{theorem: NormalizerDecompLie}, to
give normalizer decompositions of $\Uofp$ and~$\SUofp$. 
Finally, in Section~\ref{sec:AZ} we leverage the results of Section~\ref{sec:U(p) SU(p)} to give normalizer decompositions for the ``exotic" \AZ spaces.

\subsection*{Acknowledgements.}
The first author was supported by NSF grant  DMS-2204357 and by an AWM-NSF mentoring travel grant to work with the fourth author.
The second author was partially supported by Spanish State Research Agency project PID2020-116481GB-I00, 
the Severo Ochoa and María de Maeztu Program for Centers and Units of Excellence in R$\&$D (CEX2020-001084-M), and the CERCA Programme/Generalitat de Catalunya. The first, second and fourth authors acknowledge the support of the program ``Higher algebraic structures in algebra, topology and geometry" at the Mittag-Leffler Institute in Spring~2022. 

We thank the organizers of the Women in Topology III workshop, where this work was begun, as well as the Hausdorff Research Institute for Mathematics, where the workshop was held. 
The Women in Topology III workshop was supported by NSF grant DMS-1901795, the AWM ADVANCE grant NSF HRD-1500481, and Foundation Compositio Mathematica. The fifth author was also supported for the workshop by a Cheryl Praeger Travel Grant.

We are extremely grateful to Dave Benson for tutorials on group theory regarding Sections \ref{section: boring group theory} and~\ref{sec:U(p) SU(p)}. Without this help, the results would have been much less tidily packaged. Needless to say, any remaining errors are our own. 

\section{Fusion systems and linking systems}
\label{section: fusion systems}

We give a brief overview of the notion of a \pdash local compact group given in the work of Broto, Levi and Oliver \cite{BLO-Discrete}. Recall that a \defining{\pdash toral group} is an extension of a torus, $\left(S^1\right)^r$, by a finite \pdash group. We work with a discrete version of \pdash toral groups. As usual, let $\Zpinfinity$ be the union of
the cyclic \pdash groups $\integers/p^n$ under the standard inclusions.

\begin{definition}\label{def:fusion_system}
A \defining{discrete \pdash toral group} is a group $P$ given by an extension
\[
1\longrightarrow \left(\Zpinfinity\right)^r \longrightarrow P\longrightarrow \pi_{0}P\longrightarrow 1,
\]
where $r$ is a nonnegative integer and $\pi_{0}P$ is a finite \pdash group.
The \defining{identity component} of $P$ is 
$P_{0}\definedas \left(\Zpinfinity\right)^r$, and 
we call $r$ the \defining{rank} of~$P$. 
We call $\pi_{0}P$ the \defining{set of components} of~$P$.
\end{definition}

\noindent Note that the identity component of a discrete \pdash toral group is well defined because it is the characteristic subgroup consisting of infinitely \pdash divisible elements.

\begin{definition}
\label{newdefinition:subgroup ordering}
We define $\size(P)$ of a discrete \pdash toral group $P$ as the pair $\size (P)=(r, c)$, where $r$ is the rank of $P$ and $c$ is the order of $\pi_0 P$, equipped with the 
lexicographic ordering (see \cite[A.5]{CLN}).
\end{definition}

\begin{lemma}\label{newlemma:order-is-height}
If $P\to P'$ is a monomorphism of discrete \pdash toral groups, then $\size(P)\leq \size(P')$, with equality if and only if $P\to P'$ is an isomorphism.
\end{lemma}

\begin{proof}
See \cite[after Definition 1.1]{BLO-Discrete}.
\end{proof}

Given two discrete \pdash toral groups $P$ and~$Q$, let $\Hom (P,Q)$
denote the set of group homomorphisms from $P$ to~$Q$. 
If $P$ and $Q$ are subgroups
of a larger group~$S$, then $\Hom_S(P,Q)$ denotes
the set of those homomorphisms (necessarily monomorphisms) induced by conjugation by
elements of~$S$.

The following definition is a straightforward generalization of the definition of fusion systems over finite \pdash groups (see \cite{BLO-Finite}).

\begin{definition}\cite[Defn.~2.1]{BLO-Discrete}
\label{definition: SaturatedFusion}
A \defining{fusion system} $\F$ over a discrete \pdash toral group $S$ is a
subcategory of the category of groups, defined as follows. The objects of
$\Fcal$ are all of the subgroups of~$S$. The morphism sets $\Hom_{\F}(P,Q)$
contain only group monomorphisms, and satisfy the following conditions.
\begin{itemize}
\item[(a)] $ \Hom_S(P,Q) \subseteq \Hom_{\F}(P,Q)$ for all $ P,Q\subgroupeq S $. In particular, all subgroup inclusions are in~$\Fcal$.
\item[(b)] Every morphism in $\Fcal$ factors as the composite of an isomorphism in $\Fcal$ followed by a subgroup inclusion.
\end{itemize}
\end{definition}

The same language of ``outer automorphisms" is used for fusion systems as for groups. In particular, just as
$\Out_{S}(P)\definedas\Aut_{S}(P)/\Aut_{P}(P)$, we define
$\Out_{\Fcal}(P)\definedas\Aut_{\Fcal}(P)/\Aut_{P}(P)$. 
In addition, we say that two subgroups $P,P'$ of~$S$ are 
\defining{$\Fcal$-conjugate} if there is an isomorphism $P\cong P'$ in~$\Fcal$. 

In order for a fusion system to have good properties and model conjugacy relations among \pdash subgroups of a group, it must  satisfy an extra set of axioms, for ``saturation." The definition is given in \cite[\S2]{BLO-Discrete}, and we refer the reader to this source, as the definition is fairly long and technical, and we do not need to use any of the details. 

\begin{example}      \label{example: fusion system for G}
We recall the fusion system
$\Fcal_S(G)$ that arises from a compact Lie group~$G$
(\cite[\S9]{BLO-Discrete}). 
Fix a choice of maximal torus~$\Tbold\subgroupeq G$. Let $W\definedas N_G(\Tbold)/\Tbold$ denote the Weyl group, and fix 
a Sylow \pdash subgroup $W_{p}\subgroupeq W$. Let $\Sbold$ denote the inverse image of $W_{p}$ in~$N_G(\Tbold)$. Then $\Sbold$ is a maximal \pdash toral subgroup of~$G$, unique up 
to $G$-conjugacy, and given by an extension
\[
1\longrightarrow \Tbold\longrightarrow\Sbold\longrightarrow W_{p}\longrightarrow 1.
\]
A maximal \emph{discrete} \pdash toral subgroup of $G$ is obtained by taking 
a maximal discrete \pdash toral subgroup $S$ of~$\Sbold$. All such choices are conjugate in~$\Sbold$ (\cite[proof of Prop.~9.3]{BLO-Discrete}), so $S$ necessarily contains the (unique) maximal discrete \pdash toral subgroup $T$ of~$\Tbold$, giving an extension
\[
1\longrightarrow T \longrightarrow S \longrightarrow W_{p}\longrightarrow 1.
\]
The \defining{fusion system of~$G$}, denoted $\Fcal_S(G)$, has as its object set all subgroups of~$S$, and for $ P,Q \subgroupeq S$, the morphisms are
$\Hom _{\Fcal_S(G)}(P,Q)\definedas \Hom _G(P,Q)=N_G(P,Q)/C_G{P}$.
\end{example}

The fusion system associated to a compact Lie group has the right technical property to be tractable. 

\begin{proposition}\cite[Prop.~8.3]{BLO-Discrete}
\label{proposition:out-finite}
If $G$ is a compact Lie group with maximal discrete \pdash toral subgroup~$S$, then
the fusion system $\Fcal_S(G)$ is saturated.
\end{proposition}

In general, a fusion system over a discrete \pdash toral group $S$ will have an infinite number of isomorphism classes of objects (unless $S$ is finite). Fortunately, it turns out to be possible to restrict one's attention to a smaller number of objects. 
The concepts of ``$\Fcal$-centric" and ``$\Fcal$-radical" play analogous roles in the theory of \pdash local compact groups to their group-theoretic counterparts. 

\begin{definition}\label{definition:Fcentric-Fradical}
Let $\Fcal$ be a fusion system over a discrete \pdash toral group~$S$.
\begin{enumerate}
\item \label{item: discrete Fcentric}
     A~subgroup $P\subgroupeq S$ is called \defining{$\Fcal$-centric}
     if $P$ contains all elements of $S$ that centralize it,
     and likewise all $\Fcal$-conjugates of $P$ contain their
     $S$-centralizers.
\item \label{item: discrete Fradical}
     A subgroup $P\subgroupeq S$ is called \defining{$\Fcal$-radical}
     if $\Out_{\Fcal}(P)=\Aut_{\Fcal}(P)/\Aut_{P}(P)$
     contains no nontrivial normal \pdash subgroup.
\end{enumerate}
\end{definition}

\begin{proposition}\cite[Cor.~3.5]{BLO-Discrete}
\label{proposition: BLO finite number cr}
In a saturated fusion system $\Fcal$ over a discrete \pdash toral group~$S$, there are only finitely many conjugacy classes of $\Fcal$-centric $\Fcal$-radical subgroups. 
\end{proposition}

The saturated fusion system $\Fcal_S(G)$ of the group $G$ does not contain
enough information about $G$ to recover $\pcomplete{BG}$. 
For example, if $G$ is a finite \pdash group, the fusion system can only detect
$G/Z(G)$.
We recall the definition of a centric linking system, a category associated to a saturated fusion system, 
whose nerve is mod~$p$ equivalent to~$BG$.
Details on properties of linking systems can be found in the appendix of~\cite{BLO-LoopSpaces}. We begin with the transporter category.

\begin{definition}   \label{definition: transporter category}
If $G$ is a group and $\Hcal$ is a collection of subgroups of~$G$, 
the \defining{transporter category} for~$\Hcal$, denoted $\Tcal_\Hcal(G)$, is the category
whose object set is~$\Hcal$, and whose morphism sets are given by
\[
\Hom_{\Tcal_\Hcal(G)}(P,Q)
\definedas
\left\{ 
g\in G \suchthat{gPg^{-1}\subgroupeq Q}
\right\}.
\]
If $S$ is a subgroup of $G$ and $\Hcal$ is the set of \emph{all}
subgroups of~$S$, then we write $\Tcal_S(G)$ for the corresponding transporter category.
\end{definition}

\begin{definition}
\cite[Defn.~4.1]{BLO-Discrete}
\cite[Defn.~1.9]{BLO-LoopSpaces}
\label{definition: linking}
Let $\Fcal $ be a fusion system over a discrete \pdash toral group~$S$ and let $\Hcal$ be the collection of $\Fcal$-centric subgroups. 
A \defining{centric linking system associated to~$\Fcal$} is a category $\L$ whose objects
are the subgroups in $\Hcal$, together with a pair of functors
\[
\Tcal_\Hcal(S)\xlongrightarrow{\delta} \Lcal \xlongrightarrow{\pi} \Fcal
\]
such that each object is isomorphic (in~$\L$) to one that is fully centralized in~$\Fcal$, and such that the following conditions are satisfied.

\begin{itemize}
\item[(A)] The functor $\delta$ is the identity on objects, and $\pi$ is the inclusion on objects. For each pair of objects $ P,Q \in \Hcal $,
the centralizer $Z(P)$ acts freely on $\Hom _{\L}(P,Q)$ by precomposition through~$\delta$,
and
$\pi_{P,Q}$ induces a bijection
\[
\Hom _{\Lcal}(P,Q)/Z(P)
\xrightarrow{\ \cong\ } 
\Hom _{\Fcal}(P,Q).
\]
\item[(B)] For each $P,Q\in \Hcal$ and each $ g \in N_S(P,Q)$,
the map induced by the functor $\pi_{P,Q} \colon \Hom_{\L}(P,Q)\rightarrow \Hom _{\Fcal}(P,Q)$ sends the element
$\delta _{P,Q}(g) \in \Hom_{\Lcal}(P,Q)$ to $c_g \in \Hom _{\Fcal}(P,Q)$.

\item[(C)] For each $ f \in \Hom _{\Lcal}(P,Q) $ and each $g \in P$, the following square in~$\L$ commutes:
\[
\diagram
P \rto^{f} \dto_{\delta_{P}(g)}
      & Q \dto^{\delta_{Q}(\pi(f)(g))}
\cr P \rto^f & Q.
\enddiagram
\]
\end{itemize}
\end{definition}


\begin{definition}\cite[Defn.~4.2]{BLO-Discrete}     \label{definition: p-local compact group}
A \defining{\pdash local compact group} is a triple $(S,\Fcal,\Lcal)$, where $\Fcal$ is
a saturated fusion system over the discrete \pdash toral group~$S$,
and $\Lcal$ is a centric linking system associated to~$\Fcal$. The
\defining{classifying space} of $(S,\Fcal,\L)$ is defined as~$\pcomplete{\realize{\Lcal}}$.
\end{definition}

Broto, Levi, and Oliver, in \cite[\S 9]{BLO-Discrete}, proved that a compact Lie group~$G$ gives rise to a \pdash local compact group $(S, \Fcal_S(G),\Lcal_S(G))$ by giving a specific construction of $\Lcal_S(G)$; they then prove that it gives a suitable model for the \pdash completion of~$BG$. 

\begin{theorem}\cite[Thm.~9.10]{BLO-Discrete}\label{theorem:Lpcom=BGpcom}
Let $G$ be a compact Lie group, and fix a maximal discrete \pdash toral subgroup $S\subgroupeq G$. Then there exists a centric linking system $\Lcal_S(G)$ associated to $\Fcal_S(G)$ such that
$\pcomplete{\realize{\L_S(G)}}\simeq \pcomplete{BG}$.
\end{theorem}

Thus any theorem about decompositions of $\realize{\Lcal}$ for a \pdash local compact group
will also apply to Lie groups. For computational purposes, we will use
an alternative, more concrete model for $\Lcal_S(G)$ also described in \cite{BLO-Discrete}, which we detail in Section~\ref{sec:Lie2}.

A key result in the theory of fusion systems is the existence and uniqueness (up to equivalence)
of centric linking systems associated to a given saturated fusion system. For
saturated fusion systems over a finite \pdash group~$S$,
the result was proven first in \cite{Chermak} using the new theory of localities.
Another proof was given in \cite{Oliver-ExistenceL}  using the obstruction theory developed in \cite{BLO-Finite}. Later,  \cite{LL-ExistenceL} extended the result to saturated fusion systems over discrete \pdash toral groups.

\begin{theorem}\cite{LL-ExistenceL}
\label{theorem: LL-ExistenceL}
Let $\Fcal$ be a saturated fusion system over a discrete \pdash toral group. Up to equivalence, there exists a unique centric linking system associated to~$\Fcal$.
\end{theorem}


\section{The normalizer decomposition for \pdash local compact groups}
\label{sec:abstract-result}

Throughout this section, we assume a fixed \pdash local compact group $(S,\Fcal,\Lcal)$ (Definition~\ref{definition: p-local compact group}).
We establish a normalizer decomposition that expresses the uncompleted nerve $\realize{\Lcal}$ as a homotopy colimit
indexed on a finite poset of chains. 
We start by introducing chains and their
automorphisms. Next we prove that automorphism groups of chains in~$\Lcal$ are
virtually discrete \pdash toral (Definition~\ref{definition: virtually ptoral}).
Lastly, we prove the general, abstract normalizer decomposition result for a \pdash local compact group (Theorem~\ref{theorem: abstract decomposition theorem}), which mostly proceeds analogously to the \pdash local finite group case
in \cite{Libman-normalizer}. 

We begin in the fusion system. A chain in~$\Fcal$ is given by a sequence 
$\Pbb=(P_0\subgroupeq P_1\subgroupeq \ldots \subgroupeq P_k)$ of subgroups of~$S$. 
A chain is \defining{proper} if the inclusions are all strict. 
If $\Pbb'$ has the same length as~$\Pbb$, we say that $\Pbb$ and $\Pbb'$ are \defining{$\Fcal$-conjugate} 
if there exists an isomorphism $f\in \Hom_\Fcal\left(P_k,P'_k\right)$ such that $f\left(P_i\right)=P'_i$.

\begin{definition}   \label{definition: automorphisms of chains in L}
Let $\Pbb = (P_0 \subgroupeq\ldots\subgroupeq P_k)$ be a chain of $\Fcal$-centric subgroups of~$S$. We define $\Aut_{\Fcal}\Pbb$ as the group of $\Fcal$-automorphisms of $P_k$ that restrict to an automorphism of $P_i$ for each $0\leq i<k$. 
\end{definition}

We would like to define $\Aut_\Lcal(\Pbb)$ for a chain $\Pbb$, but first we need an analogue of the  canonical subgroup inclusions used to define $\Aut_\Fcal(\Pbb)$. It is possible to construct compatible ``distinguished inclusions" in~$\Lcal$ with the property that they project to the subset inclusions in~$\Fcal$
via $\pi\colon\Lcal\rightarrow\Fcal$ 
(Definition~\ref{definition: linking}). 

\begin{lemma}   \cite[Remark~1.6]{JLL}
\label{lemma: distinguished morphisms in L}
Let $\L$ be a centric linking system associated to a saturated fusion system $\Fcal$ on a discrete \pdash toral group~$S$. There is a coherent collection of morphisms
$\left\{
    \inclusion{P}{Q}\in \Hom_\Lcal(P,Q) \
    \suchthat{\largestrut P\subgroupeq Q}
\right\}$
with the following properties. 
\begin{enumerate}
\item $\pi\big(\inclusion{P}{Q}\big)$ is the inclusion morphism $P\subgroupeq Q$ in~$\Fcal$. 
\item $\inclusion{P}{P} = \Id$.
\item If $P\subgroupeq Q\subgroupeq R$ are subgroups of~$S$, then $\inclusion{Q}{R} \circ \inclusion{P}{Q} = \inclusion{P}{R}$.
\end{enumerate}
\end{lemma}

The spaces in our decomposition of $\realize{\Lcal}$ will be 
classifying spaces of automorphism groups in the linking system. 

\begin{definition}\cite[Def.~1.4]{Libman-normalizer}   
\label{definition: automorphisms of chains in L}
Let $\Pbb = (P_0 \subgroupeq\ldots\subgroupeq P_k)$ be a chain of $\Fcal$-centric subgroups of~$S$. 
Define $\Aut_\Lcal(\Pbb)$ to be the subgroup of 
$\prod_{i=0}^{k}\Aut_{\Lcal}P_i$ consisting of sequences $(f_i)$ satisfying $f_{i+1}\circ\iota_{P_{i}}^{P_{i+1}}=\iota_{P_{i}}^{P_{i+1}}\circ f_i$. That is, each element of $\Aut_\Lcal(\Pbb)$ gives a 
commutative ladder
\begin{equation}   \label{diagram: ladder of chains}
\begin{gathered}
\xymatrix{
P_0\ar[r]^-{\iota_{P_0}^{P_1}}\ar[d]^\isom_{f_0}
   & P_1 \ar[r]^-{\iota_{P_1}^{P_2}}\ar[d]^\isom_{f_1}
   & \dots\ar[r]^-{\iota_{P_{k-2}}^{P_{k-1}}} 
   & P_{k-1}\ar[r]^-{\iota_{P_{k-1}}^{P_k}}\ar[d]^-\isom_{f_{k-1}}
   & P_k\ar[d]^-\isom_{f_k}
\\
P_0\ar[r]_-{\iota_{P_0}^{P_1}} 
   & P_1 \ar[r]_-{\iota_{P_1}^{P_2}}
   & \dots\ar[r]_-{\iota_{P_{k-2}}^{P_{k-1}}} 
   & P_{k-1}\ar[r]_-{\iota_{P_{k-1}}^{P_k}} 
   & P_k.
}
\end{gathered}
\end{equation} 
\end{definition}

We pause for two basic lemmas that come up in lifting from the fusion system to the linking system, and in checking uniqueness properties. First, a factoring lemma follows directly from the axioms of a centric linking system.

\begin{lemma} \label{lemma: lifting chains}
Given a diagram in $\Fcal$ on the left, and a lift $\psiwiggle$ of $\psi$ to~$\Lcal$, there is a unique lift $\phiwiggle$ of $\phi$ making the diagram on the right commute in~$\Lcal$.
\[
\xymatrix{
P\ar[r]^-\subseteq\ar[d]_\phi & P'\ar[d]^\psi
\\Q\ar[r]^-\subseteq & Q'
}
\hspace{60pt}
\xymatrix{
P\ar[r]^-{\iota_{P}^{P'}}\ar[d]_\phiwiggle & P'\ar[d]^\psiwiggle
\\Q\ar[r]^-{\iota_{Q}^{Q'}} & Q'
}
\]
\end{lemma}
\begin{proof}
This follows from applying \cite[Lemma~4.3]{BLO-Discrete} to $P\xrightarrow{\phi} Q\into Q'$.
\end{proof}

It follows from the factoring properties above that morphisms in a centric linking system have good categorical properties. 

\begin{lemma}\cite[Cor.~1.8]{JLL}     
\label{lemma: categorical morphism properties in L}
Morphisms in a centric linking system are both categorical monomorphisms and categorical epimorphisms. 
\end{lemma}

Returning to the study of automorphism groups of chains, we are able to relate the automorphism groups in the linking system to those in the fusion system. 

\begin{lemma}\label{lemma: AutL to AutF surjective}
Let $\Pbb = (P_0 \subgroupeq\ldots\subgroupeq P_k)$ be a chain of $\Fcal$-centric subgroups. There is a short exact sequence
\[
1 \longrightarrow Z(P_k)
  \longrightarrow \Aut_\Lcal(\Pbb)
  \xrightarrow{\ \pi\ } \Aut_\Fcal(\Pbb)
  \longrightarrow 1.
\]
\end{lemma}

\begin{proof}
The lemma follows from Lemma \ref{lemma: lifting chains}. Given an $\Fcal$-automorphism of~$\Pbb$, we lift $P_k\rightarrow P_k$ to~$\Lcal$, with $Z(P_k)$ choices for the lift 
by Definition~\ref{definition: linking}(A) because $P_k$ is $\Fcal$-centric. Lemma \ref{lemma: lifting chains} then guarantees unique compatible lifts to all of the smaller subgroups. 
\end{proof}

\begin{lemma}    \label{lemma: chain monos}
The natural maps $\Aut_{\Lcal}(\Pbb)\rightarrow\Aut_{\Lcal}(P_i)$ are monomorphisms. 
\end{lemma}

\begin{proof}
The lemma is an immediate consequence of Lemma~\ref{lemma: categorical morphism properties in L}. An automorphism of a larger subgroup restricts uniquely (via the distinguished inclusions) to a smaller subgroup. And any particular element of 
$\Aut_{\Lcal}P_i$ may not extend to automorphisms of larger subgroups to give a commuting diagram~\eqref{diagram: ladder of chains}, but if it does, then the extension is unique. 
\end{proof}

We relate automorphism groups of chains in a linking system to virtually discrete \pdash toral groups, which were studied in the context of linking systems 
in \cite{LL-ExistenceL} and~\cite{Molinier}. 

\begin{definition}   \label{definition: virtually ptoral}
A \defining{virtually discrete \pdash toral group} is a discrete group that
contains a normal discrete \pdash torus of finite index.
\end{definition}

Like discrete \pdash toral groups, virtually discrete \pdash toral groups have good inheritance properties. 

\begin{lemma}\label{lemma:virtual-subgroup}
If $G$ is a virtually discrete \pdash toral group and $H\subgroupeq G$, then $H$ is also a virtually discrete \pdash toral group.
\end{lemma}

\begin{proof}
Let $P\triangleleft G$ be a normal discrete \pdash toral subgroup of $G$ of finite index, and let $T_P$ be the identity component of~$P$. Then $T_P\triangleleft G$ and $[G:T_P]$ is finite. 

Let $T_H$ denote the subgroup of $T_P\cap H$ consisting of infinitely \pdash divisible elements. Because $(T_P\cap H)\triangleleft H$ and $T_H$ is a characteristic subgroup of $T_P\cap H$, we have $T_H\triangleleft H$. The result follows because $H/T_H\subgroupeq G/T_P$, and the latter is finite. 
\end{proof}

Automorphism groups of chains in $\Lcal$ take values in virtually discrete \pdash toral groups. 

\begin{lemma}\label{lemma: Aut_L}
Let $\Pbb$ be a chain of subgroups in~$\Lcenrad$. Then 
$\Aut_\Lcal(\Pbb)$ is a virtually discrete \pdash toral group. 
\end{lemma}

\begin{proof}
We first establish the result for a single $\Fcal$-centric group~$P$. 
By Definition~\ref{definition: linking}(C) (with $P=Q$), the distinguished monomorphism $\delta_{P}$ identifies $P$ with
a normal subgroup of $\Aut_{\Lcal}(P)$.
We have a ladder of short exact sequences
\[
\xymatrix{
1 \ar[r]& Z(P)\ \ar@{=}[d]\ar@{^{(}->}[r]
       & P \ar@{^{(}->}[d]^{\delta_{P}}\ar@{->>}[r]
       & \Aut_{P}(P)\ar@{^{(}->}[d]\ \ar[r]& 1\\
1 \ar[r]& Z(P)\ \ar@{^{(}->}[r]
       & \Aut_{\Lcal}(P) \ar@{->>}[r]^{\pi}
       & \Aut_{\Fcal}(P) \ar[r]& 1,\\
}
\]
and the cokernel of the right-hand column is $\Out_\Fcal(P)$ (by definition). Hence we have a short exact sequence of groups
\begin{equation}   \label{eq: ses Aut Out}
0\longrightarrow P
 \longrightarrow \Aut_{\Lcal}(P)
 \longrightarrow \Out_\Fcal(P)
 \longrightarrow 0,
\end{equation}
where $P\triangleleft\Aut_\Lcal(P)$, and $\Out_\Fcal(P)$ is finite
by \cite[Prop.~2.3]{BLO-Discrete}. Since $\Aut_{\Lcal}(P)$ is an extension of a finite group by a discrete \pdash toral group, $\Aut_{\Lcal}(P)$ is virtually discrete \pdash toral. 

The result follows for chains from Lemmas \ref{lemma: chain monos}
and~\ref{lemma:virtual-subgroup}.
\end{proof}

With automorphism groups in place, we are ready to discuss the indexing category of the normalizer decomposition, following the work of \Slominska. We adapt the proof of \cite[Thm.~5.1]{Libman-normalizer}.

\begin{definition}[{\cite{Slominska-hocolim}, \cite[\S4]{Libman-normalizer}}]
\label{definition: heighted EI}
A category $\Acal$ is an \emph{E-I category} if all endomorphisms in $\Acal$ are
isomorphisms, and $\Acal$ is \defining{heighted} if 
there is a function
$h:\ob(\Acal)\to \naturals$ such that $\Hom_\Acal(A,B)\neq\emptyset$ implies
 that $h(A)\leq h(B)$, with equality if and only if $A\cong B$ in~$\Acal$.
\end{definition}

Let $\Fcenrad$ (resp. $\Lcenrad$) denote the full subcategory of~$\Fcal$ (resp.~$\Lcal$) consisting of the subgroups of~$S$ that are both $\Fcal$-centric and $\Fcal$-radical. Recall that a chain of subgroups is ``proper" if all of the inclusions are strict. 

In Lemma~\ref{lemma:L-EI} we check that $\Lcenrad$ has the structure of Definition~\ref{definition: heighted EI}.

\begin{lemma}\label{lemma:L-EI}
$\Lcenrad$  has a finite number of isomorphism classes of objects, and 
is a heighted E-I category with height function $\size(\whatever)$ in Definition~\ref{newdefinition:subgroup ordering}. 
\end{lemma}

\begin{proof}
Finiteness follows from Proposition~\ref{proposition: BLO finite number cr}.
Definition~\ref{newdefinition:subgroup ordering} gives height function because projection to $\Fcal$ takes all morphisms in $\Lcal$ to group monomorphisms of discrete \pdash toral groups, which must then have non-decreasing heights. Equality is achieved only for group isomorphisms, which lift to isomorphisms in~$\Lcal$. 
\end{proof}

\begin{definition}
\label{def:sdbar}
The poset category $\sd(\Fcenrad)$ has objects given by $\Fcal$-conjugacy
classes $[\Pbb]$ of proper chains $\Pbb$ of objects of~$\Fcenrad$. There is a morphism 
$[\Pbb]\rightarrow [\Pbb']$ if and only if $\Pbb'$ is $\Fcal$-conjugate to a chain given by a subset of~$\Pbb$.
\end{definition}

The abstract ``normalizer decomposition theorem'' expresses $\realize{\Lcal}$
as a homotopy colimit over the finite poset $\sd\Fcenrad$. 

\begin{theorem}\label{theorem: abstract decomposition theorem}
\AbstractDecompThmText
\end{theorem}

\begin{proof}
By \cite{WIT-centric-radical}, the map induced by the inclusion $|\Lcenrad|\rightarrow|\Lcal|$ is a homotopy equivalence. Hence it suffices to prove that there is a functor
$\delta\colon \sd\Fcenrad\longrightarrow \Top$ with an
equivalence
$\hocolim_{\sd(\Fcenrad)} \delta \longrightarrow \left|\Lcenrad\right|$
and a natural equivalence
$\BAut_{\Lcal}(\Pbb)\longrightarrow \delta([\Pbb])$ for each chain~$\Pbb$. 
The proof of \cite[Thm~5.1]{Libman-normalizer} applies to $\Lcenrad$ as written,
because $\Lcenrad$ is a finite heighted
E-I category by Lemma~\ref{lemma:L-EI}.
The second statement of the theorem is proved in
Lemma~\ref{lemma: Aut_L}.
\end{proof}

\begin{remark}     \label{remark: collapse trick}
If the maximal torus $T$ happens to be $\Fcal$-centric and $\Fcal$-radical, there is a simplification available 
for the indexing category in
Theorem~\ref{theorem: abstract decomposition theorem}. 
Suppose that  
$\Pbb=(T\subgroupneq P_1\subgroupneq \ldots\subgroupneq P_k)$ 
is a proper chain of $\Fcal$-centric and $\Fcal$-radical subgroups. 
Because $T$ is a characteristic subgroup of 
each of the~$P_i$, there is an isomorphism
\begin{equation}     \label{eq: shorten chain with torus}
\Aut_{\Fcal}(P_1\subgroupneq \ldots\subgroupneq P_k)
   \cong \Aut_{\Fcal}(T\subgroupneq P_1\subgroupneq \ldots \subgroupneq P_k)
\end{equation}
and Lemma~\ref{lemma: AutL to AutF surjective} gives an isomorphism
\begin{equation}     \label{eq: shorten chain with torus}
\Aut_{\Lcal}(P_1\subgroupneq \ldots\subgroupneq P_k)
   \cong \Aut_{\Lcal}(T\subgroupneq P_1\subgroupneq \ldots \subgroupneq P_k)
\end{equation}
(even an equality, if one uses Lemma~\ref{lemma: chain monos}
to regard both sides as subgroups 
of $\Aut_{\Lcal}(P_k)$). If the indexing poset $\sd\Fcenrad$ is 
not too complicated, one may be able to collapse the two corresponding nodes in the diagram. 
We use this trick in 
Section~\ref{sec:U(p) SU(p)} in our computations for $\Uofp$ (see \eqref{eq: W shape}
versus \eqref{diagram: for Up}, where we have collapsed the arrow $\Sbold\bto(\Tbold\subgroup \Sbold)$, and we use it again for the 
\AZ \pdash compact groups in Section~\ref{sec:AZ}.
\end{remark} 


\section{Application to compact Lie groups}
\label{sec:Lie2}

In this section, we study the application of
our abstract normalizer decomposition
(Theorem~\ref{theorem: abstract decomposition theorem}) 
to the case of \pdash local compact groups that arise from compact Lie groups (Example~\ref{example: fusion system for G}). 
Recall that the decomposition for 
$\left|\L\right|$ in Theorem~\ref{theorem: abstract decomposition theorem} is given in terms of $\BAut_{\L}(\Pbb)$ for proper chains 
$\Pbb=\left(P_0\subsetneq \dots\subsetneq P_k\right)$
of subgroups that are $\F$-centric and 
$\F$-radical. 
There are similar concepts in the theory of compact Lie groups. 

\begin{definition}\label{definition:pcentric-pradical}
Let $G$ be a compact Lie group with a \pdash toral subgroup~$\Pbold$. 
\begin{enumerate}
\item \label{item: pcentric}
     $\Pbold$ is \defining{\pdash centric} in~$G$
     if $\Pbold$ is a maximal \pdash toral subgroup of~$C_G(\Pbold)$. 
\item \label{item: pstubborn}
     $\Pbold$ is \defining{\pdash stubborn} in~$G$ 
     if $N_{G}\Pbold/\Pbold$ is finite and
     contains no nontrivial normal \pdash subgroup.
\end{enumerate}
\end{definition}

The following theorem is the main result for this section. It recovers a 
version of the normalizer decomposition for compact Lie groups that was described by Libman in \cite[\S1.4]{Libman-Minami}. 
Our approach via \pdash local compact groups has the advantage that 
we do not need to address the delicate issues that were studied in \cite[\S5]{Libman-Minami} for the purpose of applying Quillen's Theorem~A in a topological setting.

Let $\Rcal$ denote the collection of \pdash toral subgroups of $G$ that are both \pdash centric and \pdash stubborn in~$G$, and let $\sd\Rcal$ be the poset of $G$-conjugacy classes of chains
of proper inclusions of subgroups in~$\Rcal$.

\begin{theorem}    \label{theorem: NormalizerDecompLie}
\NormalizerDecompLieText
\end{theorem}

The proof is at the end of the section and goes through several steps. First, with regard to the indexing category, we have the following result from a previous work. 

\begin{theorem}\cite[Thm.~4.3]{WIT-normalizers}
\label{theorem: find chains}
Let $\Sbold$ be a maximal \pdash toral subgroup of a compact Lie group~$G$, and let $S$ be a maximal discrete \pdash toral subgroup $S\subgroupeq\Sbold$. The closure map $P\mapsto\Pbold$ defines an injective map of conjugacy classes of chains
\begin{equation*}
\begin{gathered}
\xymatrix{
\left\{\strut P_0\subgroupeq \ldots \subgroupeq P_k\subgroupeq S
\suchthat{
      \textrm{all $P_i$ are \fusion-centric and \fusion-radical}}\right\}
      /G\ar[d]\\
\left\{\strut \Pbold_0\subgroupeq \ldots \subgroupeq\Pbold_k \subgroupeq\Sbold
\suchthat{
       \textrm{all $\Pbold_i$ are \pdash toral,
       \pdash centric, and \pdash stubborn}}\right\}
       /G.
}
\end{gathered}
\end{equation*}
The map is a one-to-one correspondence if $\pi_0G$ is a \pdash group.
\end{theorem}

\begin{remark}     \label{remark: pi0 hypothesis}
If $\pi_0 G$ is not a \pdash group, one can still use Theorem~\ref{theorem: NormalizerDecompLie} to identify the mod~$p$ homotopy type of the functor values in the normalizer decomposition (Theorem~\ref{theorem: abstract decomposition theorem}), but 
one uses the image of the map in 
Theorem~\ref{theorem: find chains} as the indexing category, rather than $\sd\Rcal$. 
The codomain of Theorem~\ref{theorem: find chains} is the starting point. A~finite number of checks are necessary to see if \pdash centric and \pdash stubborn subgroups of $G$ have maximal discrete \pdash toral subgroups that are $\Fcal_S(G)$-radical to determine the actual indexing category. (See the proof of \cite[Thm.~4.3]{WIT-normalizers}.)
\end{remark}

The remainder of this section is devoted to establishing the weak mod~$p$ equivalence
$\BAut_{\Lcal}(\Pbb)
\simeq
    \B\!\big(\ints_i N_G(\Pbold_i)\big)$
of Theorem~\ref{theorem: NormalizerDecompLie}.
In particular, the strategy is to establish a zigzag of natural mod~$p$ equivalences of functors of chains of subgroups (with the leftmost one being an equivalence by Theorem~\ref{theorem: abstract decomposition theorem}):
\begin{equation}\label{diagram:zigzag-goal} 
\begin{gathered}
\xymatrix{
&\BAut_{\Linkwiggle_S(G)}(\Pbb) 
\quad 
\ar[d]_-{\simeq_p}
\ar[r]^-{\simeq_p}
&
B\big(\bigcap_{i=1}^k N_G(P_i)\big)
\ar[d]_-{\simeq_p}
\\
\delta([\Pbb]) &
\largestrut
\BAut_{\Lcal_S(G)}(\Pbb)\ar[l]^-{\simeq}
\quad
&
B\big(\bigcap_{i=1}^k N_G(\Pbold_i)\big).
}
\end{gathered}
\end{equation}
The auxiliary category $\Linkwiggle_S(G)$ is a variant of the 
transporter system for~$G$ and is used in
the construction of $\Lcal_S(G)$ in \cite[Prop.~9.12]{BLO-Discrete}.

We begin with the left vertical arrow of 
diagram~\eqref{diagram:zigzag-goal}, which takes the bulk of the section. 
In addition to an abstract existence result in for a linking system associated to~$\Fcal_S(G)$
(Theorem~\ref{theorem:Lpcom=BGpcom}), 
there is a construction in \cite[\S9]{BLO-Discrete} of a more direct model for $\Lcal_S(G)$ starting from the transporter category (Definition~\ref{definition: transporter category}). 
A difficulty in the construction is that the axioms of a linking system require $\Hom_{\Fcal}(P,Q)$ to be the orbits of a free $Z(P)$-action on $\Hom_{\Lcal}(P,Q)$. For this reason, the transporter system $\Tcal_S(G)$ itself cannot directly provide the linking system: 
getting to 
$\Hom_{\Fcal_S(G)}(P,Q)$ from 
$\Hom_{\Tcal_S(G)}(P,Q)=N_G(P,Q)$ would require taking the orbits by the action of the \emph{entire} centralizer~$C_GP$, which in general contains elements of finite order prime to~$p$ and elements of infinite order.
The solution is to look at successive quotients of~$\Tcal_S(G)$ (following \cite[p.~398]{BLO-Discrete}). We begin with a technical lemma. 

\begin{lemma}     \label{lemma: finite order subgroup} 
Given an $\Fcal_S(G)$-centric subgroup $P \subgroupeq S$,
the elements of $C_{G}(P)$ with finite order prime to~$p$ form a normal subgroup of $C_{G}(P)$.
\end{lemma}

\begin{proof}
Let $\Pbold$ denote the closure of~$P$ in~$G$. Because $C_G(P)=C_G(\Pbold)$, 
we may as well assume that $P$ is a maximal discrete \pdash toral subgroup 
of~$\Pbold$, that is, that $P$ is ``snugly embedded" in the sense of 
\cite[\S9]{BLO-Discrete}. The first part of the proof of \cite[Prop.~4.6]{WIT-normalizers}
establishes that $\Pbold$ is \pdash centric in~$G$. 
Hence $C_G(P)/Z(\Pbold)$ has no elements of order~$p$, and must be finite group of order prime to~$P$, call it~$F'$. 

Because $\Pbold$ is \pdash toral, $Z(\Pbold)$ is the product of a torus and a finite \pdash group. Let $T_{p'}$ denote the subgroup of $Z(\Pbold)$ consisting of elements of finite order prime to~$p$, all of which are found in the torus. Setting $Q\definedas C_G(\Pbold)/T_{p'}$, we have a map of central extensions
\begin{equation*}
\xymatrix{
1 \ar[r] & \largestrut T_{p'} \ar[r]\ar@{^(->}[d]
     & C_G(\Pbold) \ar[r]\ar@{=}[d]
     & Q \ar[r]\ar@{->>}[d]
     & 1
\\
1 \ar[r] & Z(\Pbold)  \ar[r]
     & C_G(\Pbold)\ar[r]
     & F' \ar[r]\ar@/_0.6pc/@{-->}[u]_-{s}
     & 1.
}
\end{equation*}
Then $\ker(Q\epi F')= Z(\Pbold)/T_{p'}$ is central 
in $Q=C_G(\Pbold)/T_{p'}$, and further, this kernel is the product
of a \pdash torsion group and a rational vector space. 
As a result, $Q$ is a split central extension of $F'$ and we have a section
$s\colon F'\rightarrow Q$. 
The preimage 
of $s(F')\triangleleft Q$ in $C_G(P)$ is normal and contains all elements of order prime to~$p$. 
\end{proof}

Lemma~\ref{lemma: finite order subgroup} tells us that the elements of $C_{G}(P)=C_G(\Pbold)$ of order prime to~$p$ form a subgroup that we denote by~$\nu'_P$. 

\begin{corollary}  \label{corollary: rational vector space}
For an $\Fcal_S(G)$-centric subgroup 
$P \subgroupeq S$, the cokernel of the map
$Z(P)\times\nu'_P\rightarrow C_G(P)$ is a rational vector space.  
\end{corollary}

Note that $\nu'_P$ is functorial in $\Fcal_S(G)$-centric subgroups~$P$. 
Further, since $Z(P)$ is centralized 
by $\nu'_P\subgroupeq C_{G}(P)$,
we can define a functor 
$(\Zcal\times\nu')(P)\definedas Z(P)\times \nu'_{P}
\subgroupeq C_G(P)$, consisting of all elements of finite order. 
There is a quotient map 
\[
\Tcal_S(G)/(\Zcal\times\nu')\xrightarrow{f_\infty} \Fcal_S(G)
\]
that takes the quotient of $N_G(P,Q)/(Z(P)\times\nu'_P)$ 
by the action of the rational vector space
$C_G(P)/(Z(P)\times\nu'_P)$. A rigidification argument 
\cite[Lemma 9.11]{BLO-Discrete} shows that $f_\infty$ 
admits a functorial section~$s$; that is, 
it is possible to choose compatible splittings of the rational 
vector spaces
$C_G(P)/(Z(P)\times \nu'_{P})$ into~$C_G(P)$.

\begin{definition}  \cite[\S9]{BLO-Discrete} \label{definition: Linkwiggle}
Given a section $s$ of $f_\infty$, the categories $\Lcal_S(G)$ and $\Linkwiggle_S(G)$ are defined as successive
pullbacks in the following diagram: 
\begin{equation}        \label{diagram: pullback diagrams}
\begin{gathered}
\xymatrix{
\Linkwiggle_S(G)\ar[r]\ar[d] 
   & \Lcal_S(G) \ar[d]\ar[r] 
   & \Fcal_S(G) \ar[d]^-s
\\
\Tcal_S(G)\ar[r]_-{f_{\nu'}} 
   & \Tcal_S(G)/\nu'\ar[r]_-{f_\Zcal}
   &\Tcal_S(G)/(\Zcal \x \nu')\ar@<.8ex>[u]^-{f_\infty}.
}
\end{gathered}
\end{equation}
\end{definition}

\begin{proposition}\cite[Prop.~9.12]{BLO-Discrete}
\label{proposition: LSG is a linking system}
$\Lcal_S(G)$ is a centric linking system associated to ~$\Fcal_S(G)$.
\end{proposition}

Before continuing to general properties of~$\Linkwiggle$, we call out 
a special case that we will use for $G=\SUofp$ in Sections \ref{sec:U(p) SU(p)} and~\ref{sec:AZ}.

\begin{lemma}         \label{lemma: Linking=transporter}
If $P\subgroupeq S$ satisfies $C_G(P)=Z(P)$, then 
there is a natural identification
\[
\Hom_{\Lcal_S(G)}(P,Q)\cong \Hom_{\Tcal_S(G)}(P,Q).
\]
\end{lemma}

\begin{proof}
Under our assumption, $\nu'_P$ is trivial.
In diagram~\eqref{diagram: pullback diagrams}, we first consider the fusion system $\Fcal_S(G)$, where we have 
\begin{align*}
\Hom_{\Fcal_S(G)}(P,Q)
  &=N_G(P,Q)/C_G(P)\\
  &=N_G(P,Q)/Z(P)\\
  &=\Hom_{\Tcal_S(G)/(\Zcal\times\nu')}(P,Q).
\end{align*} 
Hence $f_\infty$ is the identity
on morphism sets with $P$ as the domain, and so is~$s$. 
The middle vertical arrow becomes an isomorphism on morphism sets with $P$ as the domain.
%
%
Likewise $f_{\nu'}$ is the identity on morphism sets with $P$ as the domain, which establishes the lemma. 
\end{proof}

We resume our discussion of the general relationship of $\Linkwiggle$ to~$\Lcal$ by checking that $\Linkwiggle$ still has the good categorical properties of~$\Lcal$.

\begin{lemma}     \label{lemma: Linkwiggle epi and mono}
All morphisms in the category $\Linkwiggle_S(G)$ are categorical monomorphisms and categorical epimorphisms. 
\end{lemma}

\begin{proof}
The result follows from the fact that $\Linkwiggle_S(G)$ is a pullback category of two categories that both have the desired properties (by Proposition~\ref{proposition: LSG is a linking system} and Lemma~\ref{lemma: categorical morphism properties in L} for $\Lcal_S(G)$, and by direct computation for $\Tcal_S(G)$). 
\end{proof}




With the technical elements in hand, we 
consider the left vertical map in 
diagram~\eqref{diagram:zigzag-goal}, comparing automorphism groups of chains in $\Linkwiggle$ to those in~$\Lcal$. If
$\Pbb = (P_0 \subgroupeq\dots \subgroupeq P_k)$ is a chain of $\Fcal_S(G)$-centric groups, then an 
element of $\Aut_{\Lcal}(\Pbb)$ is a diagram such as~\eqref{diagram: ladder of chains}, and uses the distinguished inclusion morphisms of~$\Lcal$
(Lemma~\ref{lemma: distinguished morphisms in L}). 
We define automorphisms of chains in $\Linkwiggle_S(G)$ in the same way, beginning with distinguished inclusions in~$\Linkwiggle_S(G)$. 

\begin{definition}
\hfill
\begin{enumerate}
\item If  $P\subgroupeq Q\subgroupeq S$ are $\Fcal_S(G)$-centric,
the ``distinguished inclusion" 
$\iotawiggle\,_P^Q\in\Hom_{\Linkwiggle_S(G)}(P,Q)$ is defined by 
\[
\iotawiggle\,_P^Q\definedas(\iota_P^Q,e)\in \Hom_{\Lcal}(P,Q)\times N_G(P,Q).
\]
\item 
If $\Pbb = (P_0 \subgroupeq \ldots \subgroupeq P_k)$ is a chain of $\Fcal_S(G)$-centric subgroups, then an element of 
$\fwiggle=(\fwiggle_i)\in\Aut_{\Linkwiggle_S(G)}(\Pbb)$ 
is a commuting ladder
\[
\xymatrix{
P_0\ar[r]^{\iotawiggle}\ar[d]^\isom_{\fwiggle_0}
   & P_1 \ar[r]^{\iotawiggle}\ar[d]^\isom_{\fwiggle_1}
   & \ldots\ar[r]^{\iotawiggle} 
   & P_{k-1}\ar[r]^{\iotawiggle}\ar[d]^-\isom_{\fwiggle_{k-1}}
   & P_k\ar[d]^-\isom_{\fwiggle_k}
\\
P_0\ar[r]^{\iotawiggle} 
   & P_1 \ar[r]^{\iotawiggle}
   & \ldots\ar[r]^{\iotawiggle} 
   & P_{k-1}\ar[r]^{\iotawiggle} 
   & P_k.
}
\]
\end{enumerate} 
\end{definition}

\begin{proposition}\label{proposition: tildeL=L}
Let $G$ be a compact Lie group and $\Pbb$ be a chain of
$\F_S(G)$-centric $\F_S(G)$-radical subgroups. Then the induced map
$\BAut_{\Linkwiggle_S(G)}(\Pbb)\rightarrow \BAut_{\L_S(G)}(\Pbb)$ from the pullback diagram \eqref{diagram: pullback diagrams}  is a mod~$p$ equivalence.
\end{proposition}

\begin{proof}
By Lemma~\ref{lemma: chain monos}, the natural map 
$\Aut_{\Lcal}(\Pbb)\rightarrow \Aut_{\Lcal}(P_0)$ is a monomorphism. Further, $\Aut_{\Linkwiggle}(\Pbb)\rightarrow \Aut_{\Linkwiggle}(P_0)$
is likewise a monomorphism, by the same proof as 
that of Lemma~\ref{lemma: chain monos} (but using  
Lemma~\ref{lemma: Linkwiggle epi and mono} in place of 
Lemma~\ref{lemma: categorical morphism properties in L}).

Let 
$K_\Pbb\definedas 
\ker\big[\Aut_{\Linkwiggle}(\Pbb)\rightarrow \Aut_{\Lcal}(\Pbb)\big]$
and
$K_0\definedas
\ker\big[\Aut_{\Linkwiggle}(P_0)\rightarrow \Aut_{\Lcal}(P_0)\big]$.
We have a commuting diagram of short exact sequences 
\begin{equation}    \label{diagram: Lwiggle to Lcal}
\begin{gathered}
\xymatrix{
1 \ar[r] & K_{\Pbb}  \ar[d]\ar[r]
   & \Largestrut\Aut_{\Linkwiggle}(\Pbb)  \ar[r] \ar@{^(->}[d]
   & \largestrut\Aut_{\Lcal}(\Pbb)  \ar@{^(->}[d]\ar[r]
   &1
\\
1 \ar[r] & K_0\ \ar[r]
     & \ \Aut_{\Linkwiggle}(P_0)  \ar[r]
     &\Aut_{\Lcal}(P_0)\ar[r]
     &1.
}
\end{gathered}
\end{equation}

By pullback diagram \eqref{diagram: pullback diagrams}, we know that $K_0\cong\nu'_{P_0}$, a finite group of order prime to~$p$. The map $K_\Pbb\rightarrow K_0$ is a monomorphism, 
so $K_\Pbb$ is finite of order prime to~$p$ as well. 
The lemma follows by using the Serre spectral sequence for mod~$p$ homology of the classifying spaces. 
\end{proof}

Next we look at the top row of 
diagram~\eqref{diagram:zigzag-goal}.
We need to compare $\Aut_{\Linkwiggle_S(G)}(\Pbb)$ with the intersection of normalizers $\bigcap N_G(P_i)$. 

\begin{proposition}     \label{proposition: Autchain=norm}
Let $G$ be a compact Lie group and let $\Pbb$ be a chain of 
$\Fcal_S(G)$-centric $\Fcal_S(G)$-radical subgroups. 
Then the inclusion 
$\Aut_{\Linkwiggle}(\Pbb)\subgroupeq\Aut_{\Tcal_S(G)}(\Pbb)
    = \bigcap_i N_G(P_i)$ 
induces a mod~$p$ equivalence of classifying spaces.
\end{proposition}

\begin{proof}
We use pullback diagram~\eqref{diagram: pullback diagrams},
which we reproduce here, applied to automorphism groups:
\begin{equation*}       
\begin{gathered}
\xymatrix{
\Aut_{\Linkwiggle_S(G)}(\Pbb)   \ar[r]\ar[d] 
   & \Aut_{\Lcal_S(G)}(\Pbb)    \ar[d]\ar[r] 
   & \Aut_{\Fcal_S(G)}(\Pbb)    \ar[d]^-{s(\Pbb)}
\\
\Aut_{\Tcal_S(G)}(\Pbb)   \ar[r]_-{f_{\nu'}(\Pbb)} 
   & \Aut_{\Tcal_S(G)/\nu'}(\Pbb)   \ar[r]_-{f_\Zcal(\Pbb)}
   &\Aut_{\Tcal_S(G)/(\Zcal \times \nu')}(\Pbb)   
      \ar@<.8ex>[u]^-{f_\infty(\Pbb)}.
}
\end{gathered}
\end{equation*} 

Because $s$ is a section, the two ways around the outside of the rectangle  from  
$\Aut_{\Linkwiggle_S(G)}(\Pbb)$ to $\Aut_{\Fcal_S(G)}(\Pbb)$ commute. 
Further, if $\Pbb=(P_0\subgroupeq\ldots\subgroupeq P_k)$, then
$\ker\big[\Aut_{\Tcal_S(G)}(\Pbb)\rightarrow \Aut_{\Fcal_S(G)}(\Pbb)\big]\cong C_G(P_k)$, because the kernel consists of elements of $G$ that act trivially on all subgroups in~$\Pbb$. On the other hand, from the pullback diagrams, 
$\ker\big[\Aut_{\Linkwiggle_S(G)}(\Pbb)\rightarrow \Aut_{\Fcal_S(G)}(\Pbb)\big]$ is a subgroup of $C_G(P_k)$, identified by the pullback diagram as $Z(P_k)\times\nu'_{P_k}$. 

Then we have a homotopy commutative diagram of horizontal fibrations
\begin{equation}\label{eq:Ltil-vs-ints-N} 
\begin{gathered}
\xymatrix{
B\big(Z(P_k)\x \nu'_{P_k}\big)\ar[r]\ar[d] 
   & \BAut_{\Linkwiggle_S(G)}(\Pbb)\ar[d]  \ar[r]
   & \BAut_{\F_S(G)}(\Pbb)\ar@{=}[d]
\\
BC_G(P_k)\ar[r]
   & \BAut_{\Tcal_S(G)}(\Pbb)   \ar[r]
   & \BAut_{\F_S(G)}(\Pbb).
}
\end{gathered}
\end{equation}

Consider the leftmost vertical map. For an $\Fcal_S(G)$-centric subgroup~$P$, the Lyndon-Hochschild-Serre spectral sequence for mod~$p$ homology corresponding to the short exact sequence 
\[
1\longrightarrow Z(P)\times \nu'_P
 \longrightarrow C_G(P)
 \longrightarrow C_G(P)/(Z(P)\times \nu'_P)
 \longrightarrow 1
\]
collapses because $C_G(P)/(Z(P)\times \nu'_P)$ is a rational vector space (Corollary~\ref{corollary: rational vector space}). 
Thus the map of fibers in \eqref{eq:Ltil-vs-ints-N} is a mod~$p$ equivalence.
Comparing the Serre spectral sequences for mod~$p$ homology for the fibrations in \eqref{eq:Ltil-vs-ints-N}
gives the desired result.  

\end{proof}



\begin{proof}[Proof of Theorem~\ref{theorem: NormalizerDecompLie}]
In diagram~\eqref{diagram:zigzag-goal} we
have zigzag  natural transformations that are mod~$p$ equivalences by 
Proposition~\ref{proposition: tildeL=L} (left vertical) and Proposition~\ref{proposition: Autchain=norm} (top row). The right vertical map in \eqref{diagram:zigzag-goal} is a mod~$p$ equivalence by our previous work (\cite{WIT-normalizers} Thm~5.1, supported by Lemma~4.4). 

Theorem~\ref{theorem: find chains} says exactly that if $\pi_0 G$ is a \pdash group, then there is an isomorphism of posets $\sd(\Fcenrad)\cong\sd(\Rcal)$ induced by taking every subgroup $P$ to its closure
$\Pbold$ in~$G$. 
\end{proof}


\section{Group-theoretic ingredients for $\Uofp$ and $\SUofp$}
\label{section: boring group theory}

In this section, we make technical preparations for 
Section~\ref{sec:U(p) SU(p)}, where we apply our Lie group decomposition result (Theorem~\ref{theorem: NormalizerDecompLie}) to the examples of $\Uofp$ and~$\SUofp$. Application of Theorem~\ref{theorem: NormalizerDecompLie}
requires knowing the normalizers of chains of \pdash centric and \pdash stubborn subgroups, so this section is devoted to detailed computations of 
such normalizers in~$\Uofp$ and $\SUofp$. Results for $\SUofp$ are obtained by intersecting the results for $\Uofp$ with~$\SUofp$. We focus on $\Uofp$ and derive the results for $\SUofp$ at the end of the section. The results in this section 
are all known to experts. In particular, we are grateful to Dave Benson for tutorials in Bonn.  

Because we have to compute normalizers of chains by 
intersecting normalizers of subgroups, we give very specific representations. Let $\Sigma_p$ act on the right of $\{1,2,\ldots,p\}$.

\begin{definition}   
\label{definition: dramatis personae}
Let $\zeta$ be a fixed $p$-th root of unity, 
and let $\sigma\in\Sigma_p$. 
\begin{enumerate}
\item $\Sigma_p\subgroup\Uofp$ is represented by permutation matrices: $\sigma_{i,j}=1$ if $(i)\sigma=j$. 
\item  $A$ is the $p\times p$ diagonal matrix with $A_{ii}=\zeta^{i-1}$.
\item  $B$ represents the \pdash cycle $(1,2,\ldots,p)$, that is, 
$B_{i,i+1}=1$. 
\end{enumerate}
\end{definition}

Definition~\ref{definition: dramatis personae} identifies the specific representations we use of the subgroups of~$\Uofp$ required in Section~\ref{sec:U(p) SU(p)} for the normalizer decomposition of~$\Uofp$. 
We fix the subgroup of diagonal matrices as our choice of maximal torus $\Tbold$ of~$\Uofp$, and classical results establish that $N_{\Uofp}\Tbold\cong\Tbold\rtimes\Sigma_p$. A maximal \pdash toral subgroup $\Sbold$ is given by $\Tbold\rtimes\integers/p$, with  normalizer $N_{\Uofp}\Sbold\cong\Tbold\rtimes N_{\Sigma_p}\integers/p$. 

\begin{definition}\hfill
\label{definition: Sbold, Tbold, Gammabold}
\begin{enumerate}
\item $\Sbold\definedas\Tbold\rtimes\langle B\rangle\cong\Tbold\rtimes\integers/p$ is 
the chosen maximal \pdash toral subgroup containing~$\Tbold$.
\item $\Gammabold\subgroup\Uofp$ is the subgroup of $\Uofp$ generated by $A$, $B$ and $S^{1}=Z(U(p))$. 
\end{enumerate}
\end{definition}

The second group in Definition~\ref{definition: Sbold, Tbold, Gammabold}, $\Gammabold\subgroup\Uofp$, is the group
denoted by $\Gamma^U_p$ in \cite[Defn.~1]{Oliver-p-stubborn}. 
It is given by a central extension
\begin{equation}   \label{eq: Gammabold extension}
1\rightarrow S^1\rightarrow\Gammabold
 \rightarrow\integers/p\times\integers/p
 \rightarrow 1,
\end{equation}
where the factors of the quotient are represented by the matrices $A$ and~$B$ (Definition~\ref{definition: dramatis personae}). 
The commutator form is $[A,B]=\zeta^{-1} I$. 
By \cite[Thm.~6(ii)]{Oliver-p-stubborn} there is a short exact sequence 
\begin{equation}    \label{eq: normalizer of Gamma}
1\to \Gammabold \to N_{\U(p)}\Gammabold \to
\SLtwoFp \to 1.
\end{equation}
We note that, conceptually, the quotient in 
\eqref{eq: normalizer of Gamma} is actually 
$\Sp_2\field_p$, the group of automorphisms of the symplectic form on $\Gammabold/S^1$ given by the commutator. However,  $\Sp_2(\field_p)\cong\SLtwoFp$ and the latter expression is more convenient for us. 

To apply Theorem~\ref{theorem: NormalizerDecompLie} in Section~\ref{sec:U(p) SU(p)}, we need a group-theoretic understanding of 
$N_{\Uofp}(\Gammabold)$ 
and $N_{\Uofp}(\Gammabold\subgroupeq\Sbold)$,
and likewise of their counterparts in~$\SUofp$.
For odd primes, our best tool for understanding the relationship of $N_{\Uofp}(\Gammabold)$ 
and $N_{\Uofp}(\Gammabold\subgroupeq\Sbold)$ is to establish that 
\eqref{eq: normalizer of Gamma} is split (Proposition~\ref{proposition: splitting of symplectic}). 

Some initial ingredients are involved. We fix an odd prime~$p$. First, observe that $\SLtwoFp$ contains a central involution, namely the negative of the identity matrix. We choose a lift of this involution to $N_{\Uofp}\Gammabold$.

\begin{definition}        \label{definition: tau}
For odd primes, let $\tau \definedas (2,p)(3,p-1)\ldots(\frac{p+1}{2},\frac{p+3}{2})\in\Sigma_p$.
\end{definition}

With our conventions, $\tau$ is represented by the permutation matrix with ones in the upper left-hand corner and on the sub-antidiagonal (Definition~\ref{definition: dramatis personae}).
This involution allows us to express $N_{\Uofp}\Gammabold$ as the product of two almost disjoint pieces. 

\begin{lemma}    \label{lemma: tau}
For $p$ odd, $\Gammabold\cap C_{\Uofp}(\tau)=S^1$
and 
$N_{\Uofp}\Gammabold=\Gammabold\cdot C_{N_{\Uofp}\Gammabold}(\tau)$. 
\end{lemma}

\begin{proof}
Direct computation shows that $\tau$ acts on $\Gammabold$ by $\tau A\tau=A^{-1}$ and $\tau B\tau=B^{-1}$; in particular, $\tau\in N_{\Uofp}\Gammabold$. 
To see that $\Gammabold\cap C_{\Uofp}(\tau)=S^1$, observe that
$\tau$ acts freely on the set of non-identity elements of $\Gammabold/S^1$. 

To address generation of $N_{\Uofp}\Gammabold$, let $\Gammabold.2$ (using Atlas notation) be the subgroup of $N_{\Uofp}\Gammabold$ generated by $\Gammabold$ and~$\tau$. Because $\tau$ acts on $\Gammabold$ by an involution, there is a short exact sequence
\[
1\rightarrow\Gammabold\rightarrow\Gammabold.2\rightarrow\integers/2\rightarrow 1.
\]
Every element of $\Gammabold.2$ has the form $\gamma$ or $\tau\gamma$ for some element $\gamma\in\Gammabold$. 

To establish the lemma, we choose an arbitrary $x\in N_{\Uofp}\Gammabold$ and construct
$y\in\Gammabold$ and 
$z\in C_{N_{\Uofp}\Gammabold}(\tau)$ 
such that $x=yz$. 
Consider the relationship of $x$ to~$\tau$: 
suppose that $x \tau x^{-1}=\alpha\in\Gammabold.2$. 
Because $p$ is odd, the subgroups $S^1\cdot\tau$ and $S^1\cdot\alpha$ are Sylow $2$-toral subgroups of $\Gammabold.2$, hence conjugate in~$\Gammabold.2$. 

We can choose $\tau\gamma\in\Gammabold.2$ such that 
$c_{\tau\gamma}(S^1\cdot\alpha)=S^1\cdot\tau$. 
We would like to know that $c_{\tau\gamma}(\alpha)=\tau$. Certainly $c_{\tau\gamma}(\alpha)$ is an involution 
in $S^1\cdot\tau$, and since $S^1$ is central we easily compute 
that the available involutions are $-I$, $\tau$, and~$(-I)\tau$. By inspection the trace of $\tau$ is~$1$, while the traces of $-I$ and $(-I)\tau$ are $-p$ and~$-1$, respectively. Hence $\tau$ is the only option for $c_{\tau\gamma}(\alpha)$.

Substituting $x \tau x^{-1}$ for $\alpha$ in the equation
$(\tau\gamma)(\alpha)(\gamma^{-1}\tau) =\tau$
and simplifying gives
$x=\gamma^{-1}(\gamma x)$ as the desired expression, so $y=\gamma^{-1}\in\Gammabold$ and $z=\gamma x\in C_{N_{\Uofp}\Gammabold}(\tau)$. 
\end{proof}

Lemma~\ref{lemma: tau} tells us that to understand the structure of $N_{\Uofp}\Gammabold$, we should 
focus on the structure of $C_{N_{\Uofp}\Gammabold}(\tau)$. 

\begin{lemma}      \label{lemma: split centralizer}
For $p$ odd, there is an isomorphism 
\[
C_{N_{\Uofp}\Gammabold}(\tau)\cong S^1\times\SLtwoFp. 
\]
\end{lemma}

\begin{proof}
Because $C_{N_{\Uofp}\Gammabold}(\tau)\cap\Gammabold=S^1$
(Lemma~\ref{lemma: tau}),
the short exact sequence in \eqref{eq: normalizer of Gamma} restricts to a central extension
\begin{equation}    \label{equation: centralizer extension}
1\rightarrow S^1\rightarrow C_{N_{\Uofp}\Gammabold}(\tau)
  \rightarrow \SLtwoFp\rightarrow 1.
\end{equation}
Such extensions are classified by 
\[
H^2(\SLtwoFp ;S^1)\cong H^3(\SLtwoFp;\integers).
\]
However, $H^3(\SLtwoFp;\integers)=0$ by 
the universal coefficient theorem for cohomology because
\begin{itemize}
\item $\Hom\big(H_3(\SLtwoFp ;\integers), \integers\big)=0$ (the domain is a torsion group), and
\item the Schur multiplier $H_2(\SLtwoFp;\integers)$ 
is zero (\cite[p.52]{steinberg-chevalley}).
\end{itemize}
Hence the central extension \eqref{equation: centralizer extension} splits, as required. 
\end{proof}

\begin{proposition}   \label{proposition: splitting of symplectic}
If $p$ is odd, there is a splitting of the short exact sequence 
\begin{equation*}   
1\to \Gammabold \to N_{\U(p)}\Gammabold \to
\SLtwoFp \to 1.
\end{equation*} 
\end{proposition}

\begin{proof}
Lemma~\ref{lemma: split centralizer} allows us to choose a splitting 
\begin{align*}
\SL_{2}(\field_p)&\longrightarrow C_{N_{\Uofp}\Gammabold}(\tau)\subgroup N_{\Uofp}(\Gammabold). 
\end{align*}
\end{proof}

The intuition of Proposition~\ref{proposition: splitting of symplectic} is that $\SLtwoFp$ acts on a two-dimensional vector space over~$\field_p$, and $\Gammabold/Z(\Gammabold)\cong \field_p\times\field_p$, with basis elements given by $A$ and $B$ of Definition~\ref{definition: dramatis personae} 
(which commute once we kill the center. 
While the splitting of Lemma~\ref{lemma: split centralizer} is not constructive for the whole quotient~$\SLtwoFp$,
our next task is to compute an explicit splitting for the subgroup of 
upper triangular matrices. This explicit computation is used to find the normalizer in $\Uofp$ of the chain $(\Gammabold\subgroup\Sbold)$, the normalizer of the corresponding chain in~$\SUofp$, and related automorphism groups in the linking systems of the \AZ \pdash compact groups in Section~\ref{sec:AZ}. 
Our choice of the representations will lie not just in~$\Uofp$, but in~$\SUofp$ for purposes of those later computations.

\newcommand{\triangular}{d}
\newcommand{\diagonal}{s}

\begin{definition}     \label{definition: upper tri generators}
Let $\SLUpper$ denote the group of upper triangular matrices in~$\SLtwoFp$, let 
$\triangular\definedas\twobytwo{1}{2}{0}{1}$ and let 
$\diagonal\colon \integers/p^\times\hookrightarrow \SLtwoFp$ be the homomorphism 
$k\mapsto \diagonal_k\definedas
\twobytwo{k}{0}{0}{k^{-1}}$.
\end{definition}

A quick computation establishes the following lemma. 

\begin{lemma}
The group $\SLUpper$ is generated by $\triangular$ and 
$\{\diagonal_k\}$, which satisfy the relations $\triangular^p=I$
and $s_k\, d\, s_k^{-1}=\,d^{\,k^2}$.
\end{lemma}

The next definition sets the notation for the representation of $\SLUpper$ that we will use. 

\begin{definition}  \label{definition: define D and C p-1} 
Let $\zeta$ denote the fixed $p$-th root of unity. 
\begin{enumerate}
\item Let $D\in\SUofp$ be the $p\times p$ diagonal matrix with $D_{ii}=\zeta^{-(i-1)^2}$.
\item For $p\neq 2$, let
$\sigma\colon (\integers/p)^\times\rightarrow\Sigma_p$
be the homomorphism defined
by $(i)\sigma_k=k(i-1)+1$. 
Let $C_{p-1}\subgroupeq\SUofp$ denote the corresponding group of \emph{signed} permutation matrices. 
\end{enumerate}
\end{definition}

\begin{lemma}     \label{lemma: C and D}
The group $C_{p-1}$ 
normalizes~$\langle D\rangle$ and 
$\langle D\rangle\rtimes C_{p-1}\subgroupeq C_{\Uofp}(\tau)$.
\end{lemma}

\begin{proof}
Conjugating a diagonal matrix by a permutation matrix
(given the convention of Definition~\ref{definition: dramatis personae})
performs the permutation on the diagonal, so we obtain
\[
\big(\sigma_k D \sigma_k^{-1}\big)_{ii}= D_{(i)\sigma_k,(i)\sigma_k}
     = \zeta^{-[k(i-1)]^2}
     = \big(\zeta^{-(i-1)^2}\big)^{k^2}
     = D^{k^2}_{ii}.
\]
Another easy computation establishes that $\tau$ centralizes
both $C_{p-1}$ and $\langle D\rangle$. 
(In fact, if $\xi$ generates $\integers/p^\times$, 
then $\tau=\sigma_k$ for $k=\xi^{(p-1)/2}$.)
\end{proof}

\begin{corollary}   \label{corollary: rho is homomorphism}
For $p$ odd, there is a homomorphism 
$\rho\colon \SLUpper\rightarrow \SUofp$ defined by 
$\rho(\triangular)=D$ and $\rho(\diagonal_k)=\sigma_k$.
\end{corollary}

\begin{proof}
The matrices $\sigma_k$ and $D$ have the same order as their preimages, and the necessary conjugation relation is verified in the proof of Lemma~\ref{lemma: C and D}.
\end{proof}


We are now able to explain the relationship of upper triangular $2\times 2$ matrices to~$N_{\Uofp}\Gammabold$. 

\begin{lemma}    \label{lemma: upper triangular}
Let $p$ be odd. 
The image of $\rho: \SLUpper\rightarrow\SUofp$ normalizes~$\Gammabold$. For $M\in\SLUpper$, the 
action of $\rho(M)$ on
$\Gammabold/Z(\Gammabold)\cong\field_p\times\field_p
   \cong\langle A\rangle\times\langle B\rangle$ 
is via the linear transformation~$M$. 
\end{lemma}

\begin{proof}
The image of $\rho$ is generated by $C_{p-1}$ and~$D$. 
The matrix $\sigma_k\in C_{p-1}$ conjugates $A$ to~$A^k$ and $B$ to~$B^{(k^{-1}\mmod p)}$, so $\sigma_k$ normalizes~$\Gammabold$ and acts on the basis 
$(A,B)$ of $\Gammabold/Z(\Gammabold)$ by~$s_k$
(Definition~\ref{definition: upper tri generators}). 
We can also compute $DAD^{-1}=A$ and 
$DBD^{-1}=\zeta A^2 B$, so $D$ normalizes~$\Gammabold$
and acts on $(A,B)$ by $d\in\SLUpper$. 
\end{proof}

\begin{proposition}    \label{proposition: normalizer Gamma-S chain} 
For all primes, there is an extension 
\[
1 \longrightarrow \Gammabold
  \longrightarrow N_{\Uofp}(\Gammabold\subgroup\Sbold)
  \longrightarrow \SLUpper
  \longrightarrow 1.
\]
For $p$ odd, the extension is split by  $\rho\colon\SLUpper\hookrightarrow\SUofp\hookrightarrow\Uofp$.
\end{proposition}

\begin{proof}
Let $p$ be odd. Given Lemma~\ref{lemma: upper triangular}
we have only to verify that the image of~$\rho$ 
normalizes~$\Sbold$. 
Because $\sigma_k$ is a (signed) permutation matrix, it normalizes $\Tbold$. Recall that 
$\Sbold=\Tbold\rtimes\langle B\rangle$ 
(Definition~\ref{definition: Sbold, Tbold, Gammabold}) 
and $\sigma_{k} B\sigma_{k}^{-1}=B^{{(k^{-1}\mmod p)}}$,
and therefore $\sigma_k$ normalizes~$\Sbold$. 
Further, because $D$ is a diagonal matrix, $D\in\Sbold$ and necessarily normalizes~$\Sbold$. 
We conclude that 
$\im(\rho)\subgroupeq N_{\Uofp}(\Gammabold\subgroup\Sbold)$.

We would like to know that $\Gammabold$ and the image of~$\rho$ generate all of $N_{\Uofp}(\Gammabold\subgroup\Sbold)$.
We observe that 
$\rho(\SLUpper)\subgroupeq C_{N_{\Uofp}\Gammabold}(\tau)$ 
(Lemma~\ref{lemma: C and D}) 
and $\rho(\SLUpper)\cap S^1=1$, 
so $\rho(\SLUpper)\cap \Gammabold=1$ 
by Lemma~\ref{lemma: tau}. 
Hence $\rho$ is a monomorphism. 

Because upper triangular matrices form a maximal proper subgroup of~$\SLtwoFp$, either 
$N_{\Uofp}(\Gammabold\subgroup\Sbold)/\Gammabold\cong\SLUpper$
 or  
$N_{\Uofp}(\Gammabold\subgroup\Sbold)/\Gammabold
     \cong\SLtwoFp$.
The latter would require 
$N_{\Uofp}(\Gammabold\subgroup\Sbold)=N_{\Uofp}(\Gammabold)$, that is, 
$N_{\Uofp}\Gammabold\subgroup N_{\Uofp}\Sbold$;
in this situation, we would have a homomorphism
\[
N_{\Uofp}(\Gammabold)/\Gammabold
\rightarrow 
N_{\Uofp}(\Sbold)/\Sbold. 
\]
where the domain has order $p(p^2-1)$, the codomain has order $p-1$, and the kernel is $N_\Sbold(\Gammabold)/\Gammabold$, which is \pdash toral
(\cite[Lemma~A.3]{JMO}). This is impossible, so 
$N_{\Uofp}(\Gammabold\subgroup\Sbold)/\Gammabold\cong\SLUpper$,
as required. 

For $p=2$, we have to adjust the computation. In this special case, $N_{\U(2)}\Sbold=\Sbold$, so 
$N_{\U(2)}(\Gammabold\subgroup\Sbold)=N_{\Sbold}\Gammabold$.
We have an inclusion 
\[
N_{\Sbold}(\Gammabold)/\Gammabold
   \subgroup N_{\U(2)}(\Gammabold)/\Gammabold
   \cong \SLtwoFtwo\cong\Sigma_3.
   \]
The $2$-toral group $N_{\Sbold}(\Gammabold)/\Gammabold$ is nontrivial because $\Gammabold\subgroupneq\Sbold$. Hence 
$N_{\Sbold}(\Gammabold)/\Gammabold\cong\integers/2
\cong\SLUppertwo$. 
However, the extension is not split, because preimages of the generator of~$\integers/2$ have order~$4$. 
\end{proof}

For the prime~$2$, we define some additional specific elements in the normalizer of~$\Gammabold$ and use them to identify the appropriate groups in this case. 

\begin{definition}     \label{definition: extra elements}
Let $F$ and $H$ denote the following elements of~$\U(2)$:
\[
F\definedas{\twobytwo{e^{\pi i/4}}{0}{0}{e^{-\pi i/4}}}
\ \mbox{ and  }\ 
H\definedas\twobytwo{1/\sqrt{2}}{-1/\sqrt{2}}{1/\sqrt{2}}{1/\sqrt{2}}.
\]
\end{definition}

\begin{lemma}   \label{lemma: p=2 normalizers}
Let $p=2$. 
\begin{enumerate}
\item $\Gammabold$ and $F$ generate $N_{\U(2)}(\Gammabold\subgroup\Sbold)$, an extension of $\integers/2$ by~$\Gammabold$.
\item $\Gammabold$, $F$, and~$H$ generate $N_{\U(2)}(\Gammabold)$, an extension of $\Sigma_3$ by~$\Gammabold$.
\end{enumerate}
\end{lemma}

\begin{proof}
From Proposition~\ref{proposition: normalizer Gamma-S chain}, we have a short exact sequence
\begin{equation*}
\xymatrix{
1\ar[r] & \Gammabold\ar[r]
        & N_{\U(2)}(\Gammabold\subgroup\Sbold)\ar[r]
        & \integers/2  \ar[r]
        & 1.
}
\end{equation*}
We can check by computation that 
$F$ acts on $\Gammabold$ via the automorphism
$FAF^{-1}=A$ and $FBF^{-1}=(iI)AB$.
Proposition~\ref{proposition: normalizer Gamma-S chain} 
establishes the value of $N_{\U(2)}(\Gammabold\subgroup\Sbold)/\Gammabold$ as $\integers/2$, and direct computation establishes that $F\notin\Gammabold$ (for example because its trace is not zero) and $F^2=(iI)A\in\Gammabold$. 

Likewise, by~\eqref{eq: normalizer of Gamma} (since $\SLtwoFtwo\cong\Sigma_3$), we have a short exact sequence 
\begin{equation*}
\begin{gathered}
\xymatrix{
1\ar[r] & \Gammabold\ar[r]
        & N_{\U(2)}(\Gammabold)\ar[r]
        & \Sigma_3  \ar[r]
        & 1.
}
\end{gathered}
\end{equation*}
We can check that  
$HAH^{-1}=B$ and $HBH^{-1}=-A$, so $H$ normalizes~$\Gammabold$. Since $F$ and $H$ represent different transpositions in~$\Sigma_3$, we have found a generating set. 
\end{proof}
  
In the final part of the section, we discuss the groups and normalizers for $\SUofp$ that correspond to Propositions 
\ref{proposition: splitting of symplectic}
and~\ref{proposition: normalizer Gamma-S chain}. 
Unlike our work with~$\Uofp$, we move into a discrete setting for this part of our discussion. The reason is that
because of special circumstances in the transporter system for~$SU(p)$, it turns out that
in Section~\ref{sec:U(p) SU(p)} we will be able to work directly in the discrete setting, rather than the compact (continuous) setting we have to use for~$\Uofp$. 

\begin{definition}      \label{definition: notation for SU(p)} 
Let $p$ be an odd prime. We define the following discrete \pdash toral subgroups of~$\SUofp$.
\begin{enumerate}
\item $T$ is the set of $p\times p$ diagonal matrices of \pdash power order and determinant~$1$, a discrete \pdash torus of rank~$p-1$. 
\item $S=T\rtimes \langle B\rangle\cong T\rtimes \integers/p$ (Definition~\ref{definition: dramatis personae}).
\item
$\GammaSUp=\SUofp\cap\Gammabold$.
\end{enumerate}
\end{definition}

For an expression of $\Gamma$ as a group extension, we replace
\eqref{eq: Gammabold extension} with a central extension
in~$\SUofp$:
\begin{equation}   \label{eq: representation of Gamma}
1\rightarrow \integers/p \rightarrow\GammaSUp
 \rightarrow\integers/p\times\integers/p
 \rightarrow 1.
\end{equation}
For $p$ odd, the matrices $A$ and $B$ of Definition~\ref{definition: dramatis personae} are in $\SUofp$ and still represent generators of the factors of the quotient. 
For $p=2$, we replace $A$ and $B$ by $A'=iA$ and $B'=iB$, respectively, which are both in~$\SUofp$. The commutator form remains the same. 
Note that $\GammaSUp$ is isomorphic to the (unique) extra-special \pdash group of order $p^3$ and exponent~$p$.

With regard to other ingredients in our calculations for~$\Uofp$, 
we replace $\tau$ (Definition~\ref{definition: tau}) with 
$\tau'\definedas(-I)\tau\in\SUofp$ when $p\equiv 1\pmod 4$. 
In Definition~\ref{definition: define D and C p-1}, 
we have already arranged to have $D\in\SUofp$
and $C_{p-1}\subgroup\SUofp$.

\begin{proposition}\label{proposition: Gamma SES} 
Let $T$, $S$, $\Gamma$, and $\BorelSUp$ be as defined above.  
\begin{enumerate}
\item \label{item: normalizer GammaSUp}
There is a short exact sequence
\begin{equation*}   
1\to \GammaSUp \to N_{\SUofp}\GammaSUp \to
\SLtwoFp \to 1,
\end{equation*} 
and the sequence is split for odd primes. 
For $p=2$, $N_{\SU(2)}\Gamma\cong O_{48}$, the binary octahedral group of order~$48$. 
\item \label{item: normalizer GammaSUp in S'}
There is a short exact sequence 
\begin{equation*}   
1\to \GammaSUp \to N_{\SUofp}(\GammaSUp\subgroup S) \to
\BorelSUp \to 1.
\end{equation*} 
For odd primes, the extension is split
by $\rho\colon\SLUpper\hookrightarrow\SUofp$,
where $\rho$ is defined in Corollary~\ref{corollary: rho is homomorphism}. 
For $p=2$, 
$N_{\SUofp}(\GammaSUp\subgroup S)\cong Q_{16}$, the generalized quaternion group of order~$16$. 
\end{enumerate}
\end{proposition}

\begin{proof}
We first note that
$N_{\SUofp}(\GammaSUp)=\SUofp\cap N_{\Uofp}(\Gammabold)$, 
because $\Gammabold=\Gamma\cdot S^1$, and similarly
$N_{\SUofp}(S)=\SUofp\cap N_{\Uofp}(\Sbold)$.

For the proof of~\itemref{item: normalizer GammaSUp}, the replacement $\tau'$ for $\tau$ defined above allows the proof of Proposition~\ref{proposition: splitting of symplectic} (in particular, of Lemma~\ref{lemma: tau}) to go through as written. Therefore the splitting 
$\SLtwoFp\rightarrow N_{\Uofp}(\Gammabold)$ can be taken to land in~$\SUofp$, and since the image normalizes~$\Gammabold$, it also normalizes~$\Gamma$, the subgroup of elements of determinant~$1$. When $p=2$, it is well known that the elements of $\Gamma$ can be thought of as the corners of a unit cube in~$\complexes^2$, and the symmetry group is (by definition)~$O_{48}$. 

Continuing to \itemref{item: normalizer GammaSUp in S'}, the proofs of Lemma~\ref{lemma: upper triangular} and Proposition~\ref{proposition: normalizer Gamma-S chain} apply as written, with the adaptations that we identify
 $\Gammabold/Z(\Gammabold)$ with $\Gamma/Z(\Gamma)$ and 
 replace $\Tbold$ with~$T$, the elements of $\Tbold$ of $p$-power order and determinant~$1$. 
For $p=2$, Lemma~\ref{lemma: p=2 normalizers} likewise applies as written with the substitutions of $\SU(2)$, $\Gamma$, $S$, $iA$, and $iB$ for $\U(2)$, $\Gammabold$, $\Sbold$, $A$, and $B$, respectively. To see that 
$N_{\SUofp}(\GammaSUp\subgroup S)\cong Q_{16}$, 
recall that $F=\twobytwo{e^{\pi i/4}}{0}{0}{e^{-\pi i/4}}\in N_{\SUofp}(\GammaSUp\subgroup S)$ and $(iA)(iB)=\twobytwo{0}{-1}{1}{0}\in\Gamma$ are a well-known generating set for~$Q_{16}$. 
\end{proof}

\begin{remark}  \label{remark: p=2 recovery}
Proposition~\ref{proposition: Gamma SES} for $p=2$ 
will be used in Section~\ref{sec:U(p) SU(p)} to recover the homotopy pushout result of \cite[Thm.~4.1]{DMW1}.
\end{remark}


\section{Decompositions of $\U(p)$ and $\SU(p)$} 
\label{sec:U(p) SU(p)}

In this section we use Theorem~\ref{theorem: NormalizerDecompLie},
together with the results of Section~\ref{section: boring group theory} and \cite{WIT-normalizers}, to obtain
mod~$p$ normalizer decompositions of $\BU(p)$ and $\BSU(p)$,
which appear in Theorems \ref{theorem: U(p) decomposition} and~\ref{theorem: SU(p) decomposition}.
The results for $\BSU(p)$ for odd primes will be leveraged in Section~\ref{sec:AZ} to construct a decomposition of the \AZ exotic \pdash compact groups. 

Recall that we have chosen explicit representations of $\Tbold$, $\Sbold$, and their normalizers in~$\Uofp$ (Definitions \ref{definition: dramatis personae} and~\ref{definition: Sbold, Tbold, Gammabold}). 
The torus, $\Tbold$, is given by diagonal matrices, $\Sigma_p\subgroup\Uofp$ acts via permutation matrices, and 
$N_{\Uofp}(\Tbold)\cong\Tbold\rtimes\Sigma_p$. 
We chose $\Sbold\cong\Tbold\rtimes\integers/p$ to be generated by $\Tbold$ and the matrix $B$ representing the \pdash cycle 
$(1\ 2\ \ldots\ p)$. Further, 
$N_{\U(p)}\Sbold \cong \Tbold \rtimes N_{\Sigma_p}\Z/p$ 
\cite[Lemma~3]{Oliver-p-stubborn}, where 
$N_{\Sigma_p}\Z/p=\langle B\rangle \rtimes C_{p-1}$ 
(Definition~\ref{definition: define D and C p-1}). 

\begin{lemma}   \label{lemma: conjugacy classes in Up}
For $p\geq 5$, there are three conjugacy classes of \pdash centric and \pdash stubborn subgroups of $\Uofp$, which are represented by $\Sbold$, $\Tbold$, and $\Gammabold$. For $p=2,3$, there are only two such conjugacy classes, those of $\Sbold$ and~$\Gammabold$. 
\end{lemma}

\begin{proof}
It follows from \cite[Thm.~6]{Oliver-p-stubborn} that the stated groups are exactly the \pdash stubborn subgroups of~$\Sbold$. 
We check that they are also \pdash centric. 
The groups $\Sbold$ and $\Gammabold$ act irreducibly so their centralizers are both~$S^1$ \cite[Prop.~4]{Oliver-p-stubborn}, which is contained both $\Sbold$ and~$\Gammabold$. Lastly, $C_{\Uofp}(\Tbold)=\Tbold$, so indeed all three groups are \pdash centric. 
\end{proof}

To apply Theorem~\ref{theorem: NormalizerDecompLie}, we need to know not only 
the conjugacy classes of appropriate subgroups, but also the conjugacy classes
of chains, which could theoretically be a finer distinction. 

\begin{proposition}     \label{proposition: what are the chains}
For $p\geq 5$, there are five $\Uofp$-conjugacy classes of chains of subgroups of \pdash centric and \pdash stubborn subgroups of~$\Sbold$, represented by $\Sbold$, $\Tbold$, $\Gammabold$, $\Tbold\subgroup\Sbold$, and $\Gammabold\subgroup\Sbold$. For $p=2,3$, there are three: $\Sbold$, $\Gammabold$, and $\Gammabold\subgroup\Sbold$. 
\end{proposition}

\begin{proof}
Lemma~\ref{lemma: conjugacy classes in Up} gives us the result for a chain with a single subgroup. A chain $\Pbold_0\subgroupneq\Pbold_1$ has the form
$\Pbold\subgroupneq\Sbold$,
with $\Pbold$ conjugate in $\Uofp$ to $\Tbold$ (if $p\geq 5$) or~$\Gammabold$ (for any prime). If $\Pbold$ is conjugate to~$\Tbold$, then $\Pbold=\Tbold$, so there is a unique chain of this type. We treat $\Pbold\cong\Gammabold$ below. Lastly, since no conjugate of $\Gammabold$ is contained in~$\Tbold$ and vice-versa, there are no chains 
$\Pbold_0\subgroupneq\Pbold_1\subgroupneq\Pbold_2$ for any prime.

If $\Pbold\subgroupneq\Sbold$ with $\Pbold$ conjugate to $\Gammabold$ in~$\Uofp$, 
then we must show that the chain $\Pbold \subgroupneq \Sbold$ is conjugate to $\Gammabold \subgroupneq \Sbold$; that is, that there is some element $u\in \Uofp$ that conjugates $\Pbold$ to $\Gamma$ and $\Sbold$ to $\Sbold$. 
Our strategy is to show that the subgroup $\Pbold\cap\Tbold$ 
determines~$\Tbold$ (as $C_{\Uofp}(\Pbold\cap\Tbold)$),
and then adjust $u$ by an element of $N_{\Uofp}\Gammabold$ so that $\Pbold$ is still mapped to~$\Gammabold$, but $\Sbold$ is preserved. 

First,
we assert that $\Pbold\cap\Tbold$ is a subgroup of $\Pbold$ of index~$p$.
Since $\Pbold/(\Pbold\cap\Tbold)\subgroupeq\Sbold/\Tbold\cong\integers/p$, 
either $[\Pbold:\Pbold\cap\Tbold]=1$ or $[\Pbold:\Pbold\cap\Tbold]=p$. 
Since $\Pbold$ is not abelian, we know that  $[\Pbold:\Pbold\cap\Tbold]\neq 1$, so
$[\Pbold:\Pbold\cap\Tbold]=p$, and for the same reasons 
we know that $[\Gammabold:\Gammabold\cap\Tbold]=p$.
(Indeed,  
$\Gammabold\cap\Tbold$ is the product of $S^1$ and the group of order $p$ generated by the matrix~$A$ of Definition~\ref{definition: dramatis personae}.) 

Now suppose that $\Pbold$ is conjugate to~$\Gammabold$ by
$u\in\Uofp$. We must prove that the chain $\Pbold\subgroup\Sbold$ is $\Uofp$-conjugate to the chain $\Gammabold\subgroup\Sbold$. 
First, observe that 
\[
[\Gammabold:\Gammabold\cap\Tbold]=p=[\Pbold:\Pbold\cap\Tbold]=[\Gammabold: c_u(\Pbold\cap\Tbold)],
\]
so both $\Gammabold\cap\Tbold$ and $c_u(\Pbold\cap\Tbold)$ are subgroups of $\Gammabold$ of index~$p$. However, 
$N_{\Uofp}\Gammabold$ acts transitively on the 
set of subgroups of $\Gammabold$ of index~$p$, 
because $N_{\Uofp}\Gammabold/\Gammabold\cong\SLtwoFp$. Let $n\in N_{\Uofp}\Gammabold$ conjugate $c_u(\Pbold\cap\Tbold)$ 
to $\Gammabold\cap\Tbold$, so that $c_n\circ c_u$ conjugates
$(\Pbold, \Pbold\cap\Tbold)$ to $(\Gammabold, \Gammabold\cap\Tbold)$. 

We claim that $c_n\circ c_u$ in fact conjugates $\Sbold$ to~$\Sbold$. 
First, we assert that $C_{\Uofp}(\Pbold\cap\Tbold)$ and $C_{\Uofp}(\Gammabold\cap\Tbold)$ are both maximal tori of~$\Sbold$. This is because both $\Pbold\cap\Tbold$ and $\Gammabold\cap\Tbold$ are isomorphic to $S^1\times\integers/p$, 
where the second factor acts on $\complexes^p$ with $p$ distinct eigenvalues. 

Since $\Tbold$ centralizes $\Pbold\cap\Tbold$ and $\Gammabold\cap\Tbold$, we have
$C_{\Uofp}(\Pbold\cap\Tbold)=\Tbold=C_{\Uofp}(\Gammabold\cap\Tbold)$. 
We know that $c_n\circ c_u$ maps $C_{\Uofp}(\Pbold\cap\Tbold)$ to $C_{\Uofp}(\Gammabold\cap\Tbold)$, so we conclude that $c_n\circ c_u$ maps 
$\Tbold$ to itself, while still mapping $\Pbold$ to~$\Gammabold$. 

Lastly, $\Gammabold$ and $\Tbold$ generate~$\Sbold$ (and the same for $\Pbold$ and~$\Tbold$). Hence $c_n\circ c_u$ conjugates $\Pbold\subgroup\Sbold$ to $\Gammabold\subgroup\Sbold$, which finishes the proof. 
\end{proof}

\begin{remark}     \label{remark: choose elements in SU(p)}
For purposes of the $\SUofp$ calculation later in the section, note that we can adjust $u$ and $n$ in the proof of 
Proposition~\ref{proposition: what are the chains}
to be elements of $\SUofp$ if we wish.
\end{remark}

As a result of Lemma~\ref{lemma: conjugacy classes in Up} and Proposition~\ref{proposition: what are the chains}, we find that for $p\geq 5$, the indexing poset of Theorem~\ref{theorem: NormalizerDecompLie}
for the normalizer decomposition of~$\Uofp$ is
\begin{equation}   \label{eq: W shape}
\begin{gathered}
\xymatrix{
& (\Gammabold\subgroup \Sbold) \ar[dr]\ar[dl]
   && (\Tbold\subgroup \Sbold)\ar[dr]\ar[dl]
\\
\Gammabold && \Sbold && \Tbold,
}
\end{gathered}
\end{equation}
while for $p=2,3$ the diagram has just the left three nodes of the diagram above. 
We make one more simplification based on Remark~\ref{remark: collapse trick}. 

\begin{lemma}\label{lem:pushout-simplification Up}
For all primes, the homotopy colimit of the following diagram is mod~$p$
equivalent to $BU(p)$: 
\begin{equation}    \label{diagram: for Up}
\begin{gathered}
\xymatrix@C=5pt{
BN_{\Uofp}(\Gammabold\subgroupeq\Sbold) 
  \ar[d]\ar[rr] 
 && BN_{\U(p)}(\Tbold)
\\
BN_{\U(p)}(\Gammabold).
}
\end{gathered}
\end{equation}
\end{lemma}

\begin{proof}
We use Theorem~\ref{theorem: NormalizerDecompLie}. 
For $p\geq 5$, when the resulting diagram~\eqref{eq: W shape} has 
five nodes, we have $N_{\Uofp}(\Sbold)=N_{\Uofp}(\Tbold\subgroup\Sbold)$, which allows us to collapse that 
leg of the diagram. 

For $p=2,3$, diagram~\eqref{eq: W shape} has only three nodes, 
and the one labeled ``$\Sbold$" is assigned the classifying space of $N_{\Uofp}\Sbold$. 
However, when $p=2,3$ we have $N_{\Uofp}\Sbold=N_{\Uofp}\Tbold$, so we still obtain~\eqref{diagram: for Up}.
\end{proof}


The detailed group-theoretic calculations to identify the 
normalizers in
Lemma~\ref{lem:pushout-simplification Up} more specifically was done in Section~\ref{section: boring group theory}, and we draw on them for the following theorem.

\begin{theorem}    \label{theorem: U(p) decomposition} 
Let $\SLUpper$ denote the group upper triangular matrices in $\SLtwoFp$ and let $\Tbold$ denote the fixed maximal torus of~$\Uofp$.
\begin{enumerate}
\item Let $p$ be an odd prime. 
The homotopy pushout of the diagram below is mod~$p$ equivalent to~$BU(p)$:
\[
\xymatrix@C=5pt{
 B\left(\Gammabold\rtimes \UTriangular(\SLtwoFp)\right)
    \ar[d]\ar[rr] 
 &  & B\left(\strut\Tbold\rtimes\Sigma_p\right)
\\
B\left(\strut \Gammabold\rtimes \SLtwoFp \right) .
}
\]
\item Let $p=2$ and let $\Gammabold.\Sigma_2$ and $\Gammabold.\Sigma_3$ denote the extensions in
Lemma~\ref{lemma: p=2 normalizers}.
The homotopy pushout of the diagram below is mod~$2$ equivalent to~$BU(2)$: 
\[
\xymatrix@C=5pt{
B\left(\Gammabold.\Sigma_2\right)
      \ar[d]\ar[rr] 
   && B\left(\strut\Tbold\rtimes\Sigma_2\right)
\\
B\left(\strut \Gammabold.\Sigma_3 \right).
}
\]
\end{enumerate}
\end{theorem}

\begin{proof}
Lemma~\ref{lem:pushout-simplification Up} gives us a diagram whose homotopy colimit is mod~$p$ equivalent to~$\BU(p)$ for all primes, and the value of $N_{\Uofp}\Tbold$ is classical. 
We have made explicit compatible choices of representatives for the conjugacy classes of chains
in Proposition~\ref{proposition: what are the chains}. With these choices we can consider a functor $\BAut_\Lcal(\Pbb)$ from the poset with the same homotopy colimit as~$\delta([\Pbb])$.
For odd primes, the value $B(\Gammabold\rtimes\Borel)$ is provided by
Proposition~\ref{proposition: normalizer Gamma-S chain} and the value 
$B\left(\strut \Gammabold\rtimes \SLtwoFp \right)$ is provided by 
Proposition~\ref{proposition: splitting of symplectic}.
For $p=2$, the values are provided by Lemma~\ref{lemma: p=2 normalizers}.
\end{proof}

\begin{remark}
We could write the entries at the middle and left nodes
of Theorem~\ref{theorem: U(p) decomposition}
as the classifying spaces of central extensions of the finite groups 
$\field_p^2\rtimes \Borel$ and $\field_p^2\rtimes \SLtwoFp$, respectively, by~$S^1$. 
\end{remark}


In the rest of the section, we make adjustments to our calculation for
$\Uofp$ to give a result analogous to Theorem~\ref{theorem: U(p) decomposition}
for~$\SU(p)$. 
In particular, we
make adjustments to describe the relevant part of the linking system directly in terms of discrete \pdash toral subgroups, something that is not possible for~$\Uofp$ because $Z(\Uofp)$ is not finite.

First we need to identify the chains of subgroups that form the indexing category. 
We begin with compact groups, and identify the conjugacy classes of subgroups of $\SU(p)$ that are \pdash stubborn and \pdash centric, so as to use Theorem~\ref{theorem: find chains}.
Let 
$\TboldSUp
   \definedas\SUofp\cap\Tbold$, 
let 
$\SboldSUp
   \definedas\SUofp\cap\Sbold$, 
and let 
$\GammaSUp
   \definedas\SUofp\cap\Gammabold$. 

\begin{lemma}
\label{lemma: conjugacy classes in SUp}
For $p\geq 5$, there are three conjugacy classes of \pdash centric and \pdash stubborn subgroups of $\SUofp$, which are represented by $\SboldSUp$, $\TboldSUp$, and $\GammaSUp$. For $p=2,3$, there are only two such conjugacy classes, those of $\SboldSUp$ and~$\GammaSUp$. 
\end{lemma} 

\begin{proof}
The lemma follows from Lemma~\ref{lemma: conjugacy classes in Up}
and \cite[Thm.~10]{Oliver-p-stubborn}. 
\end{proof} 

\begin{remark}\label{remark:Gamma-structure}
The group $\TboldSUp$ is represented by diagonal matrices with determinant~$1$. 
The group $\SboldSUp\cong\TboldSUp\rtimes\Sigma_p$, where $\Sigma_p$ acts through \emph{signed} permutation matrices.
The group $\GammaSUp$ is an extra-special \pdash group 
\[
1\rightarrow \integers/p\rightarrow\GammaSUp\rightarrow \integers/p\times\integers/p\rightarrow 1,
\]
corresponding to the symplectic form. Recall that $\Gammabold$ was generated by matrices $A$ and $B$, which have determinant~$1$ (see Definition \ref{definition: dramatis personae}), together with~$S^1$. Similarly every element of $\GammaSUp$ can be written in the form $(\zeta I)A^i B^j$ where $\zeta$ is a $p$-th root of unity (replace $A$ and $B$ with $A'=iA$ and $B'=iB$ for $p=2$---see discussion following Definition~\ref{definition: notation for SU(p)}).
\end{remark}

\begin{lemma}     \label{lemma: what are the chains in SUp}
For $p\geq 5$ there are five $\SUofp$-conjugacy classes of chains of subgroups of \pdash centric and \pdash stubborn subgroups of~$\SboldSUp$; they are represented by $\SboldSUp$, $\TboldSUp$, $\GammaSUp$, $(\TboldSUp\subgroup\SboldSUp)$, and $(\GammaSUp\subgroup\SboldSUp)$. For $p=2,3$, there are three: $\SboldSUp$, $\GammaSUp$, and $(\GammaSUp\subgroup\SboldSUp)$. 
\end{lemma}

\begin{proof}
The proof of Proposition~\ref{proposition: what are the chains}
applies as written (see Remark~\ref{remark: choose elements in SU(p)}). 
\end{proof}

We have already defined maximal discrete \pdash toral subgroups $S$ and $T$ of $\SboldSUp$ and $\TboldSUp$, respectively, in Definition~\ref{definition: notation for SU(p)}, and $\Gamma$ is a finite \pdash group.

\begin{proposition}    \label{proposition: SUp discrete chains} 
The subgroup $S\subgroup\SUofp$ is a maximal discrete \pdash toral subgroup. 
The $\SUofp$-conjugacy classes of chains of 
$\Fcal_S(\SUofp)$-centric and $\Fcal_S(\SUofp)$-radical subgroups of $S$ are represented by~$S$, $\Gamma$, $(\Gamma\subgroup S)$, and for $p\geq 5$ also $T$ and $(T\subgroup S)$. 
\end{proposition}

\begin{proof}
First, $S$, $T$, and $\Gamma$ are all maximal
in their closures, which are \pdash centric and \pdash stubborn. Because $\pi_0\SUofp$ is a \pdash group, the proof given of Theorem~\ref{theorem: find chains} in \cite{WIT-normalizers} establishes that $S$, $T$, and $\Gamma$ are 
$\Fcal_S(\SUofp)$-centric and $\Fcal_S(\SUofp)$-radical. The proposition then follows from Theorem~\ref{theorem: find chains}.
\end{proof}

With the conjugacy classes of chains of $\Fcal_S(G)$-centric and $\Fcal_S(G)$-radical subgroups of~$S$ in hand, we know that the indexing category for $\SUofp$ is exactly analogous 
to~\eqref{eq: W shape}, with $S$, $T$, and $\Gamma$ replacing $\Sbold$, $\Tbold$, and~$\Gammabold$, respectively. 
Next we need to know the associated automorphism groups in the linking system. 
Unlike the situation for~$\Uofp$, for $\SUofp$ the transporter system is isomorphic to the linking system for most of the automorphism groups we need to calculate. As a result, the next lemma explicitly identifies automorphism groups in 
$\Lcal_S(\SUofp)$.

\begin{lemma} \label{lemma: Auts in link SU(p)}
Let $\Lcal=\Lcal_S(\SUofp)$. 
\begin{enumerate}
\item $\Aut_{\Lcal}(\Gamma)\cong N_{\SUofp}\Gamma$.
\item $\Aut_{\Lcal}(\Gamma\subgroup S)
           \cong N_{\SUofp}(\GammaSUp\subgroup S)$.
\end{enumerate}
\end{lemma}

\begin{proof}
To use Lemma~\ref{lemma: Linking=transporter} to identify automorphism groups in $\Lcal$ with normalizers in~$\SUofp$, we need to check that for $P=\Gamma, S$, we have $C_{\SUofp}(P)=Z(P)$. The subgroups $\Gamma$ and $S$ both act irreducibly on~$\complexes^p$ (\cite[Thm.~6]{Oliver-p-stubborn} applies), and hence $C_{\SUofp}(\Gamma)=C_{\SUofp}(S)=Z(\SUofp)\cong\integers/p$, scalar matrices of order~$p$. And indeed, this group is exactly $Z(\Gamma)=Z(S)$. 

\end{proof}

In the next theorem, where we state the normalizer decomposition for~$\SUofp$, we observe that the upper right-hand corner is $N_{\SUofp}T$. Since $\TboldSUp$ is abelian, its maximal discrete \pdash toral subgroup is unique, and in fact $N_{\SUofp}T=N_{\SUofp}\TboldSUp$.

\begin{theorem}    \label{theorem: SU(p) decomposition} 
\SUDecompositionTheoremText
\end{theorem}

\begin{proof}
For $p\geq 5$, Theorem~\ref{theorem: abstract decomposition theorem} together with 
Proposition~\ref{proposition: SUp discrete chains} 
give the following diagram as the one whose pushout is mod~$p$ equivalent to~$B\SUofp$ (with the homeomorphism of the third leg provided by Remark~\ref{remark: collapse trick}):
{\small
\begin{equation*}   
\begin{gathered}
\xymatrix{
& \BAut_{\Lcal}(\Gamma\subgroup S) \ar[dr]\ar[dl]
   && \BAut_{\Lcal}(T\subgroup S)
       \ar[dr]\ar[dl]_-{\cong}
\\
\BAut_{\Lcal}\Gamma && \BAut_{\Lcal}S && \BAut_{\Lcal}T.
}
\end{gathered}
\end{equation*}
} 
We can collapse the homeomorphism to obtain the diagram 
\begin{equation}     \label{eq: V shape SU(p)}   
\begin{gathered}
\xymatrix{
& \BAut_{\Lcal}(\Gamma\subgroup S) \ar[dr]\ar[dl]
\\
\BAut_{\Lcal}\Gamma && \BAut_{\Lcal}T
}.
\end{gathered}
\end{equation}

For $p\geq 5$, it
follows from \eqref{eq: ses Aut Out}
that $\Aut_{\Lcal}(T)\cong T\rtimes\Sigma_p$. 
For the other two nodes, we have actual equivalences to $BN_{\SUofp}\Gamma$ and 
$BN_{\SUofp}(\Gamma\subgroup S)$ given by Lemma~\ref{lemma: Auts in link SU(p)}. The values of 
$N_{\SUofp}\Gamma$ and $N_{\SUofp}(\Gamma\subgroup S)$
are given in 
Proposition~\ref{proposition: Gamma SES}, which proves the result for $p\geq 5$. 

For $p=2,3$, $T$ is not $\Fcal_S(\SUofp)$-radical, so our indexing diagram only contains $\Gamma$, $(\Gamma\subgroup S)$, and~$S$. However, for $p=2,3$ we have $N_{\SUofp}S\cong T\rtimes N_{\Sigma_p}\integers/p\cong N_{\SUofp}T$ and 
\eqref{eq: V shape SU(p)} is still the correct diagram. 
\end{proof}


\section{\AZ \pdash compact groups}\label{sec:AZ}
For all of this section, we assume that $p$ is an odd prime. 

In Theorem~\ref{theorem: abstract decomposition theorem}, we gave a normalizer decomposition for classifying spaces of arbitrary \pdash local compact groups $(S,\Fcal,\Lcal)$, and our first examples, in Section~\ref{sec:U(p) SU(p)},
came from compact Lie groups. Another important class of examples of \pdash local compact groups are those that arise from \pdash compact groups (Definition~\ref{definition: p-compact}), a class of loop spaces introduced in~\cite{dwyer-wilkerson-fixed-point} that generalizes \pdash completed classifying spaces of connected compact Lie groups. Although there may be no underlying group, there are analogues of Sylow \pdash subgroups and maximal tori, which allow one to construct a fusion system \cite[\S10]{BLO-Discrete}.

The goal of this section is to compute the normalizer decomposition for the particular \pdash compact groups constructed in \cite{aguade-modular}, also called the \AZ \pdash compact groups. There are four such spaces---one at $p=3$, two at $p=5$, and one at $p=7$. They are among the exotic examples of \pdash compact groups in the classification of \cite{AGMV-pcompact}.
The \AZ \pdash compact groups are closely related to the special unitary groups at the corresponding primes: roughly speaking, they are obtained by enlarging the Weyl group of $\SU(p)$ to certain reflection groups. Our strategy is to exploit this connection, together with our calculations for $\SU(p)$ in Section~\ref{sec:U(p) SU(p)}. In particular, the fusion systems of the \AZ \pdash compact groups have the same objects as the fusion systems of the corresponding special unitary groups. The difference between the morphism sets can be described in terms of the action of Adams operations 
(Definition~\ref{definition: adams ops}).

We describe the spaces below, and our main result is the following normalizer decomposition of these exotic \pdash compact groups. Let $\GLUpper$ denote the subgroup of upper-triangular matrices in $\GLtwoFp$. 
(The group $\Gamma$ is defined in Definition~\ref{definition: notation for SU(p)}.)

\begin{theorem}   \label{theorem: pushout for AZ} 
\AZTheoremText
\end{theorem}

We begin our discussion by reviewing the relevant definitions regarding \pdash compact groups, following \cite[\S3]{dwyer-wilkerson-fixed-point}. 

\begin{definition}   \cite[Defn.~2.3]{dwyer-wilkerson-fixed-point}
\label{definition: p-compact}
A \pdash compact group is a loop space $(X,BX)$ such that
$H^*(X;\field_p)$ is a finite $\field_p$-vector space,
together with a pointed \pdash complete space~$BX$
and an equivalence $X\xrightarrow{\simeq} \Omega BX$.
\end{definition}

If $G$ is a connected compact Lie group, then 
$(\pcomplete{G}, \pcomplete{BG})$ is a
\pdash compact group for any prime~$p$, since 
$H^*(\Omega(\pcomplete{BG});\field_p)\isom H^*(\pcomplete{G};\field_p)$  is finite. Not all \pdash compact groups arise in this way, but they do possess analogous structures to those of compact Lie groups, and we discuss
these next. The classification of \pdash compact groups in terms of Weyl group data was achieved in \cite{AGMV-pcompact,Andersen-Grodal-2compact}.

To define a \pdash local compact group associated to a \pdash compact group, one needs a notion of discrete \pdash toral subgroup. 

\begin{definition}
\hfill
\begin{enumerate}
 \item A \defining{discrete \pdash toral subgroup} $(P,i)$
 of a \pdash compact group~$X$
is a discrete \pdash toral group $P$ with a map 
$Bi\colon \pcomplete{BP}\rightarrow BX$ whose  homotopy fiber has finite mod~$p$ homology. 
\item $(T,i)$ is a \defining{maximal discrete \pdash torus} for $X$ if $T$ is a discrete \pdash torus, $(T,i)$ is a subgroup of~$X$, 
and 
for any other discrete \pdash torus $(A,j)$ of~$X$, there is a group homomorphism $f\colon A\to T$ such that $Bi\circ Bf\simeq Bj$. 
\item $(S,i)$ is a \defining{maximal discrete \pdash toral subgroup} of~$X$ (or, a \defining{Sylow \pdash subgroup} of~$X$)
if for any other discrete \pdash toral subgroup $(Q,j)$ of $X$, there is a group homomorphism $f\colon Q\to S$ such that $Bi\circ Bf\simeq Bj$. 
\end{enumerate}
\end{definition}

\begin{theorem}[{\cite[Prop.~10.1(a)]{BLO-Discrete}}]
Any \pdash compact group has a maximal discrete \pdash torus and a maximal discrete \pdash toral subgroup. 
\end{theorem}

We fix a choice~$(S,\iota)$ of maximal discrete \pdash toral subgroup of~$X$. If $X$ were a Lie group $G$, then the maps in $\Fcal_S(G)$ would be restrictions
of conjugation maps $G\to G$, which 
induce a map $BG\to BG$ homotopic to the identity. This observation motivates the definition of  
morphisms in a fusion system associated to a \pdash compact group~$X$.

\begin{definition}[{\cite[Defn.~10.2]{BLO-Discrete}}]
\label{definition:F_X}
Let $(X,BX)$ be a \pdash compact group and let $(S,\iota)$ with $B\iota\colon BS\rightarrow BX$ 
be a choice
of maximal discrete
\pdash toral subgroup.
The associated fusion system $\Fcal_X$ on $S$ has the following morphism sets: for $P,Q\subgroupeq S$, an element of
$\Hom_{\Fcal_X}(P,Q)$ consists of a group homomorphism
$f\colon P\to Q$ such that the diagram
\[
\xymatrix{
BP\ar[d]_{Bf}\ar[r]^{B\iota_P} 
    & BX
\\
BQ\ar[ru]_{B\iota_Q}
} 
\]
commutes up to homotopy, where $B\iota_P$ denotes
the composition $BP\rightarrow BS\xrightarrow{B\iota} BX$. 
\end{definition} 

\begin{theorem}[{\cite[Prop.~10.5, Thm.~10.7]{BLO-Discrete}}] 
\label{thm:p-compact fusion}
Let $(X,BX)$ be a \pdash compact group and let $S\rightarrow X$ be a choice
of a maximal discrete \pdash toral subgroup.
Then the fusion system $\Fcal_X$ on $S$ is saturated, and 
there is a centric
linking system $\Lcal_{X}$ associated to $\Fcal_X$ with 
$
\pcomplete{\realize{\Lcal_X}}\hteq BX.
$
\end{theorem}

Our interest is in the \AZ \pdash compact groups, constructed at specific
primes~$p$ by modifying $\pcomplete{\BSU(p)}$ so that the Weyl group is enlarged.

\begin{definition}
Continuing the notation of Definition~\ref{definition:F_X}, 
the \defining{Weyl group} of $BX$ is $\Aut_{\Fcal_X}(T)$ where $T$ is the maximal torus in~$S$.
\end{definition}

The enlargement of the Weyl group of $\pcomplete{\BSU(p)}$
gives a new Weyl group for the \pdash compact group which contains Adams operations, well-known self maps of a discrete \pdash torus whose definition we review next.
Let $T = \left( \Zpinfinity \right)^r$ be a discrete \pdash torus, and recall that
$\Hom(\Zpinfinity, \Zpinfinity)  
    \cong \Zphat$, the \pdash completion of the integers.  
For any $\xi\in (\pcomplete{\integers})^\x$, the corresponding diagonal group isomorphism $T\to T$ induces a self homotopy equivalence 
$\pcomplete{BT}\to \pcomplete{BT}$, which induces multiplication by $\xi$ on $H^*(BT;\Zphat)$. 
Its restriction to the maximal discrete \pdash torus $T\subgroup S$ is the $\xi$-power automorphism
(see \cite[Def.~2.3]{JLL}).
\begin{definition}    \label{definition: adams ops}
Let $(X,BX)$ a \pdash compact group. For $\xi\in (\pcomplete{\integers})^\x$, an
\defining{Adams operation $\psi^\xi\colon BX\to BX$} is a self-homotopy equivalence 
whose restriction to the maximal torus $\pcomplete{BT}$ is
homotopic to the self homotopy equivalence induced by $\xi\in \Hom(\Zpinfinity, \Zpinfinity)$.
\end{definition}

For compact connected Lie groups, there is an existence result.

\begin{theorem}\cite[Cor.~3.5]{JMO-Selfhomotopy}
\label{thm:moller}
Let $G$ be a compact connected Lie group.
For all $\xi\in
(\pcomplete{\integers})^\x$, there is an unstable Adams operation map $\psi^\xi\colon \pcomplete{BG}\to \pcomplete{BG}$.
\end{theorem}

Likewise there is an existence and uniqueness result for \pdash local compact groups. 

\begin{theorem}\cite{AGMV-pcompact,Andersen-Grodal-2compact}
 For any connected \pdash compact group, there exists exactly one unstable Adams operation of degree $\xi$ for every \pdash adic unit $\xi \in (\Zphat)^\times$. 
\end{theorem}

In the context of \pdash local compact groups, unstable Adams operations have been studied in \cite{JLL} and \cite{LL-Adams}.

We turn to describing the structure of the \AZ \pdash compact groups, the four
\pdash compact groups whose normalizer decompositions we  compute in this section. 
Let $\omega$ denote the Teichm\"uller lift
$\omega\colon \integers/(p-1)\cong (\integers/p)^\x\to \Zphat$, which identifies $(p-1)$-st roots of unity in~$\Zphat$.

\begin{definition}[{\cite[\S10, p.37]{aguade-modular}}] \label{definition:AZ}
For $i=12,29,31,34$, let $G_i$ denote the $i$-th group in the list of Shephard-Todd~\cite{Shephard-Todd}, and let $p_{12} = 3$, $p_{29} = 5$, $p_{31} = 5$, and $p_{34}=7$.
Let $BX_i$ denote the homotopy colimit of the diagram
\begin{equation} \label{eq: hocolim diagram for BX}
\xymatrix{
\pcomplete{(BT_i)}\POS!L(0.7)\ar@(ul,dl)_{G_i} 
    \ar[rr]^-{G_i/\Sigma_{p_i}}
&& \pcomplete{B\SU(p_i)}\POS!R(.7) \ar@(ur,dr)^{Z(G_i)} 
},
\end{equation}
where $T_i$ is a maximal discrete \pdash torus of $\SU(p_i)$
and the arrow labeled $G_i/\Sigma_{p_i}$ indicates that the morphism set
between the two nodes is isomorphic, as a $G_i$-set via precomposition, to the quotient~$G_i/\Sigma_{p_i}$. The group $Z(G_i)\cong \integers/(p_i-1)$ acts on $\pcomplete{B\SU(p_i)}$ via unstable Adams operations $\psi^{\omega(k)}$, and $G_i$ acts on $\pcomplete{(BT_i)}$ via its representation as a \pdash adic reflection group. 
\end{definition}

To simplify notation, we fix the value of the index~$i$ 
in Definition~\ref{definition:AZ} 
and suppress subscripts, writing simply $X$ and $BX$ for the corresponding \AZ space and its classifying space, $p$~for the prime, and $G$~for 
the Weyl group~$G_i$. 

Automorphisms of $\pcomplete{BT}$ that correspond to conjugation 
by elements of $\SU(p)$ 
act trivially on the morphism set $G/\Sigma_{p}$, since such conjugations induce maps of $\BSU(p)$ 
that are homotopic to the identity. 
On the other hand, the Adams operations act on $\pcomplete{BSU(p)}$ by maps that 
are not homotopic to the identity  \cite[Thm.~1]{JMO}.
We summarize the properties of \AZ \pdash compact groups that we need. 

\begin{theorem} [\cite{aguade-modular}] \label{thm:X_i-facts}
Let $BX$ denote one of the \AZ \pdash compact groups (Definition~\ref{definition:AZ}).
Then $BX$ is a \pdash compact group with Weyl group $G$. Moreover:
\begin{enumerate}
\item The map $\pcomplete{\BSU(p)}\to BX$ of \pdash compact groups is a monomorphism; that is, the homotopy fiber has finite mod $p$ homology.
\item If $S$ is a Sylow \pdash subgroup of~$\SU(p)$, then the composition $BS\to \BSU(p) \to \pcomplete{\BSU(p)} \to BX$ is a  Sylow \pdash subgroup of~$X$.
\item The center $Z(G) = \integers/(p-1)\subgroup G$ acts on $\pcomplete{BT}$ via Adams operations.
\end{enumerate}
\end{theorem}

Once we fix a Sylow \pdash subgroup of~$\SU(p)$, we will fix the corresponding one for $X$ using Theorem~\ref{thm:X_i-facts}(2). We then simplify the notation by writing $\Fcal_{\SU(p)}$ and $\Fcal_X$ for the corresponding fusion systems $\Fcal_S({\SU(p)})$ and~$\Fcal_S(X)$. Likewise, we write $\Lcal_{\SUofp}$ and $\Lcal_X$ for the associated linking systems. 

Let $(S,\Fcal_X)$ denote the fusion system associated to $BX$ that is provided by Theorem~\ref{thm:p-compact fusion} and has the same Sylow \pdash subgroup as the fusion system $(S, \Fcal_{\SU(p)})$ for $\SU(p)$ studied in Section~\ref{sec:U(p) SU(p)}. Our next goal is to relate the two fusion systems; the crucial input is Theorem~\ref{thm:X-SU(p)}, which says that the difference between $\Fcal_{\SU(p)}$ and $\Fcal_X$ is entirely described by the automorphisms of~$Z(S)$. 
When applied to the subgroups of interest (Corollary~\ref{corollary: chains for X}), these automorphisms consist of the Adams operations.

We need a restriction of a fusion system for the purpose of comparing 
$\Fcal_{\SU(p)}$ and~$\Fcal_X$. The following definition is specialized 
from \cite[Defn.~2.1]{BLO-LoopSpaces}.

\begin{definition}     \label{definition: centralizer fusion system}
Given a saturated fusion system $\Fcal$ over a discrete \pdash toral group~$S$,
let $C_\Fcal(Z(S))$ denote the following fusion system:
\begin{itemize}
\item objects of $C_\Fcal(Z(S))$ are subgroups $P$ with $Z(S)\subgroupeq P\subgroupeq S$
\item morphisms of $C_\Fcal(Z(S))$ are morphisms in $\Fcal$ that restrict to the identity on $Z(S)$.
\end{itemize}
\end{definition}

\begin{proposition} \cite[Thm.~2.3]{BLO-LoopSpaces}
$C_\Fcal(Z(S))$ is a saturated fusion system over~$S$.  
\end{proposition}


Definition~\ref{definition: centralizer fusion system} gives us exactly 
the concept we need to identify $\Fcal_{\SUofp}$ inside~$\Fcal_X$.

\begin{theorem}\cite[\S5.2]{cantarero-castellana} \label{thm:X-SU(p)}
Let $X$ denote an \AZ \pdash compact group with fusion system
$(S, \Fcal_X)$. 
Then there is an isomorphism of fusion systems
$$ C_{\Fcal_X}(Z(S))\cong \Fcal_{\SU(p)}.$$
Moreover, the $\Fcal_X$-centric subgroups of $S$ coincide with the
$\Fcal_{\SU(p)}$-centric subgroups of~$S$, and likewise
the $\Fcal_X$-radical subgroups of $S$ coincide with the
$\Fcal_{\SU(p)}$-radical subgroups of~$S$.
\end{theorem}

Our goal is to obtain a homotopy colimit decomposition of $BX$ using
Theorem~\ref{theorem: abstract decomposition theorem} with the same indexing
category as for $\SU(p)$ (see Lemma~\ref{lemma: what are the chains in SUp}).
For computational purposes, we use the explicit discrete \pdash toral representations from Section~\ref{sec:U(p) SU(p)}.

\begin{lemma}\label{lemma:Hcal_X}
Let $\Hcal_X$ denote the collection of $\Fcal_X$-centric, $\Fcal_X$-radical
subgroups of~$S$, and analogously for $\Hcal_{\SU(p)}$. Then
$\sd\Hcal_X = \sd\Hcal_{\SU(p)}$.
\end{lemma}

\begin{proof}
By Theorem~\ref{thm:X-SU(p)}, a subgroup $P\subseteq S$ is $\Fcal_X$-centric
and $\Fcal_X$-radical if and only if it is $\Fcal_{\SU(p)}$-centric and
$\Fcal_{\SU(p)}$-radical. If two chains are conjugate in $\Fcal_{\SU(p)}$, then they are also 
conjugate in the larger fusion system~$\Fcal_X$.
No new conjugacies of chains are possible in~$\Fcal_X$ because the chains are all group-theoretically distinct. 
\end{proof}

\begin{corollary}    \label{corollary: chains for X}
The indexing category $\sd\Hcal_X$ has objects 
$S$, $\Gamma$, $(\Gamma\subgroup S)$, and for $p\geq 5$ also $T$ and $T\subgroup S$. 
\end{corollary}

\begin{proof}
The corollary follows from Lemma~\ref{lemma:Hcal_X} and
Proposition~\ref{proposition: SUp discrete chains}.
\end{proof}

In order to use our abstract decomposition result, Theorem~\ref{theorem: abstract decomposition theorem}, we will compute
$\Aut_{\Lcal_X}(\Pbb)$ for chains $\Pbb \in \mathcal{H}_X$ by relating these groups to
$\Aut_{\Lcal_{\SU(p)}}(\Pbb)$. By a counting argument, we compute $\Aut(\Gamma)$, the full abstract automorphism group of~$\Gamma$.

\begin{lemma}     \label{lemma: abstract aut of Gamma}
The abstract automorphism group $\Aut(\Gamma)$ is isomorphic to $\Aff_2\field_p=(\field_p)^2\rtimes \GLtwoFp$. In particular, with the representation of~$\Gamma$~in \eqref{eq: representation of Gamma}, we can take the map $\Aut(\Gamma)\rightarrow\GLtwoFp$ to be the representation given by the action of $\Aut(\Gamma)$ on $\Gamma/Z(\Gamma)\cong (\field_p)^2\cong\langle A\rangle\times\langle B\rangle$; that is, $A$ and $B$ represent the basis of $(\field_p)^2$ on which $\GLtwoFp$ acts.   
\end{lemma}

\begin{proof}
As  preliminary, suppose that $f$ is an automorphism of~$\Gamma$ that fixes the center $Z(\Gamma)$ and passes to the identity 
on~$\Gamma/Z(\Gamma)\cong\langle A\rangle\times \langle B\rangle$. 
Then for some $i,j$, we have $f(A)=\zeta^i A$ and $f(B)=\zeta^j B$. 
Such automorphisms are realized for all $i$ and $j$ by inner automorphisms 
of~$\Gamma$ because $ABA^{-1}=\zeta^{-1} B$ and $BAB^{-1}=\zeta A$. Conversely, since $[A,B]\in Z(\Gamma)$, inner automorphisms of $\Gamma$ pass to the identity automorphism on~$\Gamma/Z(\Gamma)$. 

We assert that there is a short exact sequence
\begin{equation}     \label{eq: SES Aut Gamma}
\xymatrix{
1 \ar[r] & \Gamma  \ar[r]/(Z(\Gamma))
         & \Aut(\Gamma)     \ar[r]
         & \Aut(\Gamma/Z(\Gamma))  \ar[r]
         & 1.       
}
\end{equation}
The inclusion is via inner automorphisms, and 
$\Aut(\Gamma) \rightarrow \Aut(\Gamma/Z(\Gamma))$ is the natural induced map since any group automorphism must stabilize the center. 
The action of $\Gamma$ on itself by conjugation passes to the identity automorphism on $\Gamma/Z(\Gamma)$ as indicated above.

It remains to show that the second map in \eqref{eq: SES Aut Gamma} is a split epimorphism. 
Identifying $\Gamma/Z(\Gamma)\cong\field_p\times\field_p$ with 
$\langle A\rangle\times \langle B\rangle$ as above, 
we observe that for any element of
$\Aut(\Gamma/Z(\Gamma))\cong\GLtwoFp$, 
there is a unique automorphism of $Z(\Gamma)$ that causes the commutator on $\Gamma$ to be preserved, which is sufficient to give an automorphism of~$\Gamma$. Further, uniqueness guarantees that the preimages assemble into a subgroup of~$\Aut(\Gamma)$, i.e. that the short exact sequence is split by a group homomorphism. 
The lemma follows. 
\end{proof}

As in the case of~$\SUofp$, we only need to compute the automorphism groups in $\Lcal_X$ of a subset of the chains named in Corollary~\ref{corollary: chains for X}, namely $\Gamma$~and $(\Gamma\subgroup S)$, and we begin with the fusion system. 
Any automorphism of~$\Gamma$ (or of~$S$) 
restricts to an automorphism of the group's center,
$Z(\Gamma)=Z(S)=Z(\SUofp)\cong \Z/p$. 
It is helpful to call out what the fusion system automorphisms are doing to the center of these groups. 

\begin{definition}
Let $\rescen_{\Gamma}\colon \Aut_{\Fcal_X}(\Gamma)\rightarrow\Aut(\integers/p)$ by restriction to~$Z(\Gamma)$, and define  
$\rescen_{\Gamma\subgroup S}\colon \Aut_{\Fcal_X}(\Gamma\subgroup S)\rightarrow\Aut(\integers/p)$ similarly.
\end{definition}

\begin{lemma}      \label{lemma: Res chain is onto}
The homomorphisms $\rescen_{\Gamma\subgroup S}$ 
and $\rescen_\Gamma$ are surjective.
\end{lemma}

\begin{proof}
We need only prove the result for $\rescen_{\Gamma\subgroup S}$. 
Let $\xi\in\integers/p^\times\cong\Aut(\integers/p)$ be a generator; we seek an element of $\Aut_{\F_X}(\Gamma\subgroup S)$ that has order $p-1$ and whose restriction to
$Z(\SU(p))\cong \Z/p$ is $\xi\in\Aut(\Z/p)$.

Let $\psi^\xi$ be an unstable Adams operation on $\pcomplete{\BSU(p)}$ (see
Theorem~\ref{definition: adams ops}) whose restriction to $T$ is $x\mapsto x^\xi$.
By \cite[Prop.~3.5]{JLL}, $\psi^\xi$~restricts to an automorphism of $\psi^\xi\colon S\to S$ which is in $\Fcal_X$ 
by Definition~\ref{definition:F_X}. 
Possibly $\psi^\xi$ does not stabilize~$\Gamma$, so suppose that $\psi^\xi$ takes $(\Gamma\subgroup S)$ to some 
$(\Gamma'\subgroup S)$. 
By Proposition~\ref{proposition: SUp discrete chains}, there exists $x\in\SU(p)$ such that 
$c_x(\Gamma'\subgroup S)=(\Gamma\subgroup S)$. 
Then one can consider the automorphism 
$\phi\definedas c_x\circ \psi^\xi\in \Fcal_X$, which still stabilizes $S$ and also restricts to an
automorphism of~$\Gamma$. We have $\rescen(\phi)=\rescen(\psi^\xi)=\xi$ (because conjugation by an element of $\SUofp$ fixes $\Z/p=Z(\SUofp)$), and 
$\phi\in\Aut_{\Fcal_X}(\Gamma\subgroup S)$, 
showing that $\rescen_{\Gamma\subgroup S}$ is surjective.   
\end{proof}

Lemma~\ref{lemma: Res chain is onto} shows that all of the Adams operations extend to automorphisms of the chains that are relevant for our decomposition and allows us to extract the full fusion system automorphism groups. 

\begin{proposition}\label{proposition: FX SES}
Let $\GLUpper\subgroup\GLtwoFp$ denote the subgroup of upper triangular matrices. 
\begin{enumerate}
\item $\Aut_{\Fcal_X}(\Gamma)
           \cong \Aff_2\field_p=(\field_p)^2\rtimes \GLtwoFp$.
\item $\Aut_{\Fcal_X}(\Gamma\subgroup S)
    \cong (\field_p)^2\rtimes \GLUpper$.
\end{enumerate}
\end{proposition}

\begin{proof}
For the first statement, we will use a counting argument. 
By Proposition~\ref{thm:X-SU(p)}, we have $\ker(\rescen_\Gamma)=\Aut_{\Fcal_{\SU(p)}}(\Gamma)$.
Since $\rescen_{\Gamma}$ is surjective by Lemma~\ref{lemma: Res chain is onto}, we have a short exact sequence
\begin{equation}       \label{F-SU(p)-X}
1\longrightarrow \Aut_{\Fcal_{\SU(p)}}(\Gamma)
 \longrightarrow \Aut_{\Fcal_X}(\Gamma)
 \xrightarrow{\ \rescen_\Gamma\ } \Aut(\integers/p) 
 \longrightarrow 1. 
\end{equation}
We know that $\Aut_{\Fcal_{\SUofp}}(\Gamma)$, which is given by $N_{\SUofp}(\Gamma)/C_{\SUofp}(\Gamma)$, has $p^3(p^2-1)$ elements (Proposition~\ref{proposition: Gamma SES} and $C_{\SUofp}(\Gamma)=Z(\SUofp)\cong\integers/p$). 
Hence $\Aut_{\Fcal_{X}}(\Gamma)$ has
$(p-1)\cdot p^3(p^2-1)$ elements. As a result, 
Lemma~\ref{lemma: abstract aut of Gamma} tells us that
$\Aut_{\F_X}(\Gamma)$ must be the full abstract automorphism group of~$\Gamma$, i.e. 
$\Aut_{\F_X}(\Gamma) = (\field_p)^2\rtimes \GLtwoFp$, completing the proof of~(1).

For $(\Gamma\subgroup S)$ we begin in a similar way. 
By Proposition~\ref{thm:X-SU(p)}, 
$\ker(\rescen_{\Gamma\subgroup S})=\Aut_{\Fcal_{\SU(p)}}(\Gamma\subgroup S)$. By Lemma~\ref{lemma: Res chain is onto}
we have a short exact sequence 

\begin{equation}          \label{eq: SES for F_X chain} 
1\longrightarrow \Aut_{\Fcal_{\SU(p)}}(\Gamma\subgroup S)
 \longrightarrow \Aut_{\Fcal_X}(\Gamma\subgroup S)
 \xrightarrow{\ \rescen_{\Gamma\subgroup S}} \Aut(\integers/p) 
 \longrightarrow 1. 
\end{equation}

By Proposition~\ref{proposition: Gamma SES} we have
\[
\Aut_{\Fcal_{\SU(p)}}(\Gamma\subgroup S) 
\isom 
(\field_p)^2\rtimes\SLUpper,
\]
where the matrices $A$ and $B$ represent the basis of $(\field_p)^2$ and $\SLUpper$ acts on this basis in the standard way. 

The quotient group in 
\eqref{eq: SES for F_X chain} is generated by $\phi\in\Aut_{\Fcal_X}(\Gamma\subgroup S)$ constructed in 
Lemma~\ref{lemma: Res chain is onto}. 
We know that $\phi$ 
acts as $(-)^\xi$ on~$T$ (by construction) and as the identity on $S/T$ (by~\cite[Lemma 2.5]{JLL}). Hence we have
$\phi(A) = A^\xi$ and $\phi(B) = \zeta^i A^j B$ for some~$j$. Thus~$\phi$ corresponds in $\Aut_{\Fcal_X}(\Gamma)/\Gamma$ to a matrix 
$\begin{pmatrix} 
\xi&j
\\0&1
\end{pmatrix}$,
an upper triangular matrix of determinant~$\xi$. 
Thus the matrices representing $\BorelSUp$ and $\phi$ generate $\GLUpper$.
\end{proof}

Proposition~\ref{proposition: FX SES} gives the relevant automorphism groups of chains in the fusion system~$\Fcal_X$. The next task is to lift them to the linking system~$\Lcal_X$. 

\begin{proposition} \label{proposition: Aut_LX split SES}
\hfill
\begin{enumerate}
\item $\Aut_{\Lcal_X}(\Gamma) \cong \Gamma\rtimes \GLtwoFp$
\item $\Aut_{\Lcal_X}(\Gamma\subgroup S)
           \cong \Gamma\rtimes\GLUpper$.
\end{enumerate}
\end{proposition}

\begin{proof}
By Definition~\ref{definition: linking}(C),
$\Gamma$ is normal in both $\Aut_{\Lcal_X}(\Gamma)$ and $\Aut_{\Lcal_{\SUofp}}(\Gamma)$, so we compare the two quotients in diagram \eqref{eq:Aut-mod-Gamma} below.
By Lemma~\ref{lemma: Auts in link SU(p)}
and Proposition~\ref{proposition: Gamma SES}, we have 
$\Aut_{\Lcal_\SUofp}(\Gamma)/\Gamma \cong \SLtwoFp$.
By Lemma~\ref{lemma: abstract aut of Gamma} and Proposition~\ref{proposition: FX SES}, 
\begin{align*}
\Aut_{\Lcal_X}(\Gamma)/\Gamma 
    &
    \cong
[\Aut_{\Lcal_X}(\Gamma)/Z(\Gamma)]/[\Gamma/Z(\Gamma)]  \\
&\cong \Aut_{\Fcal_X}\!\big(\Gamma)/(\Gamma/Z(\Gamma)\big)\\
     & \cong \GLtwoFp. 
\end{align*}
Hence we can compare the automorphism groups in $\Lcal_{\SUofp}$ and $\Lcal_X$ with a commutative ladder of short exact sequences, where the first row is split by 
Proposition~\ref{proposition: Gamma SES}:
\begin{equation}\label{eq:Aut-mod-Gamma} 
\begin{gathered}
\xymatrix{
1 \ar[r] & \Gamma\ar@{=}[d]\ar[r] 
         & \Aut_{\Lcal_\SUofp}(\Gamma) \ar[r]\ar[d] 
         & \strut
         \SLtwoFp\ar[d]\ar[r]\ar@/_1pc/[l]  
         & 1
 \\ 1\ar[r] & \Gamma\ar[r] & \Aut_{\Lcal_X}(\Gamma)\ar[r] & \GLtwoFp\ar[r] & 1.
} 
\end{gathered}
\end{equation}
We want to show that the second row of 
\eqref{eq:Aut-mod-Gamma} 
is also split. The extension is classified by an element of
$H^2(\GLtwoFp;Z(\Gamma))
    \cong H^2(\GLtwoFp;\integers/p)$.
Further, the classifying element is in the kernel of the restriction
$H^2(\GLtwoFp;\integers/p) \rightarrow H^2(\SLtwoFp;\integers/p)$ because the top extension is split. However, the restriction map is injective, because 
$\SLtwoFp\subgroup \GLtwoFp$ is a subgroup of index prime to~$p$.
Hence the lower short exact sequence is also split,
which establishes~(1). 

For~(2), let $Z\definedas Z(\Gamma)=Z(S)$. Note that $\Aut_{\Fcal_X}(\Gamma)$ is the quotient of $\Aut_{\Lcal_X}(\Gamma)$ by~$Z$ ($=Z(\Gamma)$), and likewise by 
Lemma~\ref{lemma: AutL to AutF surjective} $\Aut_{\Fcal_X}(\Gamma\subgroup S)$ is the quotient of $\Aut_{\Lcal_X}(\Gamma\subgroup S)$ by~$Z$ ($=Z(S)$).
This gives the following commutative square:
\begin{equation}     \label{diagram: not quite a pullback}
\begin{gathered}
\xymatrix{
\Aut_{\Lcal_X}(\Gamma\subgroup S)\ 
     \ar@{^(->}[r]
     \ar@{->>}[d]^-\pi
& \Aut_{\Lcal_X}(\Gamma)
     \ar@{->>}[d]^-\pi
\\
\Aut_{\Fcal_X}(\Gamma\subgroup S) 
     \ar@{^(->}[r]
& \Aut_{\Fcal_X}(\Gamma). 
}
\end{gathered}
\end{equation}
Here the top horizontal arrow is a monomorphism by Lemma~\ref{lemma: chain monos}, and 
the bottom horizontal arrow is a monomorphism as well.
While this is not quite enough to say that \eqref{diagram: not quite a pullback} is a pullback square, it is enough to guarantee a monomorphism from 
$\Aut_{\Lcal_X}(\Gamma\subgroup S)$ into the pullback of the other three corners. We collect the values of those groups from~(1) and from
Proposition~\ref{proposition: FX SES} and compute the pullback to obtain the diagram
\begin{equation}      \label{diagram: the actual groups}
\begin{gathered}
\xymatrix{
\Aut_{\Lcal_X}(\Gamma\subgroup S)\  
     \ar@{^(->}[r]
& \Gamma \rtimes\UTriangular(\GLtwoFp)\ 
     \ar@{^(->}[r]
     \ar@{->>}[d]^-\pi
& \Gamma \rtimes\GLtwoFp
     \ar@{->>}[d]^-\pi
\\
&(\field_p)^2\rtimes\UTriangular(\GLtwoFp)
     \ar[r]
&(\field_p)^2\rtimes\GLtwoFp.
}
\end{gathered}
\end{equation}
The result now follows from a counting argument, because 
$\Aut_{\Lcal_X}(\Gamma\subgroup S)$ and 
$(\field_p)^2\rtimes\UTriangular(\GLtwoFp)$ each have $p$ times as many elements as $(\field_p)^2\rtimes\UTriangular(\GLtwoFp)$, by observation and by \eqref{diagram: not quite a pullback}, respectively. 
\end{proof}

The arguments in Propositions \ref{proposition: FX SES} and~\ref{proposition: Aut_LX split SES} cannot be generalized to compute $\Aut_{\Lcal_X}(T)$, since they used the fact that $Z(\Gamma)=Z(S)$. However, this remaining case follows easily from results in the literature.

\begin{lemma}\label{lem:Aut_LX T}
There is an isomorphism $\Aut_{\Lcal_X}(T) \cong T\rtimes G$.
\end{lemma}
\begin{proof}
From equation \eqref{eq: ses Aut Out}, we have a short exact sequence $$T\to
\Aut_{\Lcal_X}(T)\to \Out_{\Fcal_X}(T)$$ where $\Out_{\Fcal_X}(T) =
\Aut_{\Fcal_X}(T) = G$ by Theorem~\ref{thm:X_i-facts}. This exact sequence describes the normalizer of the maximal torus in the \pdash compact group, which by \cite[Theorem 1.2]{andersen-splitting} splits.
\end{proof}

\begin{proof}[Proof of Theorem~\ref{theorem: pushout for AZ}]
Combining Theorem~\ref{theorem: abstract decomposition theorem} and Theorem~\ref{thm:p-compact fusion}, we have an equivalence $BX \simeq \pcomplete{(\colim_{\sd\Hcal_X} \delta)}$, with a natural equivalence $\BAut_\Lcal(\Pbb)\to \delta([\Pbb])$. We fix a choice for conjugacy classes of chains 
(see Proposition~\ref{proposition: SUp discrete chains}).
By Lemma~\ref{lemma:Hcal_X} and Proposition~\ref{proposition: SUp discrete chains}, the indexing category $\sd\Hcal_X$ is 
\begin{equation*}
\begin{gathered}
\xymatrix{
& (\Gamma\subgroup S) \ar[dr]\ar[dl]
   && (T\subgroup S)\ar[dr]\ar[dl]
\\
\Gamma && S && T,
}
\end{gathered}
\end{equation*}
but \eqref{eq: shorten chain with torus} says that
the colimit of $\delta$  over $\sd\Hcal_X$ simplifies to a pushout 
\[
\BAut_{\L_X}(\Gamma)\longleftarrow \BAut_{\L_X}(\Gamma\subgroup S) \longrightarrow \BAut_{\L_X}(T). 
\]
The theorem now follows from the identifications of these terms in Proposition~\ref{proposition: Aut_LX split SES} and 
Lemma~\ref{lem:Aut_LX T}.
\end{proof}

\vspace{50pt}

\bibliographystyle{amsalpha}
\bibliography{WIT-fusion-bibliography}

\end{document}